\documentclass{article}
\usepackage[utf8]{inputenc}
\usepackage{amsmath, amssymb, amsthm, mathtools}
\usepackage[margin=0.5in]{geometry}
\usepackage{xcolor}

\title{Denjoy-Carleman Microlocal Regularity on Smooth Real Submanifolds of Complex Space}
\author{Antonio Victor da Silva Jr. and Nicholas Braun Rodrigues}

\newcommand{\comment}[1]{}
\newcommand{\Cl}{\mathcal{C}}
\newcommand{\V}{\mathcal{V}}
\newcommand{\M}{\mathcal{M}}
\newcommand{\Cm}{\Cl^\M}
\newcommand{\Ci}{\Cl^\infty}
\newcommand{\F}{\mathcal{F}}
\newcommand{\WF}{\mathrm{WF}}
\newcommand{\C}{\mathbb{C}}
\newcommand{\R}{\mathbb{R}}
\newcommand{\Z}{\mathbb{Z}}
\newcommand{\del}{\partial}
\newcommand{\N}{\mathrm{N}}
\newcommand{\T}{\mathrm{T}}
\renewcommand{\d}{\mathrm{d}}
\renewcommand{\Re}{\mathrm{Re}}
\renewcommand{\Im}{\mathrm{Im}}
\renewcommand{\epsilon}{\varepsilon}

\newtheorem{thm}{Theorem}[section]
\newtheorem{prob}[thm]{Problem}
\newtheorem{lemma}[thm]{Lemma}
\newtheorem{prop}[thm]{Proposition}
\newtheorem{cor}[thm]{Corollary}

\theoremstyle{definition}
\newtheorem{defn}[thm]{Definition}

\theoremstyle{remark}

\newtheorem{rmk}[thm]{Remark}

\begin{document}

\maketitle

\begin{abstract}
    We prove the existence of approximate solutions in the (regular) Denjoy-Carleman sense for some systems of smooth complex vector fields. 
    Such approximate solutions provide a well defined notion of Denjoy-Carleman wave front set of distributions on maximally real submanifolds in complex space which can be characterized in terms of the decay of the Fourier-Bros-Iagolnitzer transform. 
    We also apply the approximate solutions to analyze the Denjoy-Carleman microlocal regularity of solutions of certain systems of first-order nonlinear partial differential equations. 
\end{abstract}

\section{Introduction}

It is a well-known result that real-analytic functions have holomorphic extensions to the complex space. A similar result is valid for regular Denjoy-Carleman classes on $\mathbb{R}^m$. Given a regular sequence (see Definition \ref{def:regular_sequence}) $\mathcal{M} = (M_k)_{k = 0}^\infty$ we say that a smooth function $u$ is $\mathcal{M}$-Denjoy-Carleman, and we write $u \in \mathcal{C}^\mathcal{M}(\mathbb{R}^m)$, if for every complact set $K \subset \mathbb{R}^m$ there exists a positive constant $C > 0$ such that

\begin{equation*}
    \sup_{x \in K} |\partial^\alpha u(x)| \leq C^{|\alpha| + 1} M_{|\alpha|},
    \eqno{\forall \alpha \in \mathbb{Z}_+^m}.
\end{equation*}

\noindent In \cite{dyn'kin} Dyn'kin proved that a function $u$ belongs to $\mathcal{C}^\mathcal{M}(\mathbb{R}^m)$ if and only if for every point $p \in \mathbb{R}^m$ and relatively compact open neighborhood $U \subset \mathbb{R}^m$ of $p$, there exist an open neighborhood $\mathcal{O} \subset \mathbb{C}^m$ of $p$ with $\mathcal{O} \cap \mathbb{R}^m = U$, a function $F \in \mathcal{C}^\infty(\mathcal{O})$ and constants $C,\delta > 0$ such that

\begin{equation*}
    \begin{cases}
      F(x) = u(x);\\
      |\bar{\partial}_z F(x+iy)| \leq C^{k+1}\frac{M_k}{k!}|y|^k,
    \end{cases}
    \eqno{
    \begin{matrix}
        x \in U, & \\
        |y| \leq \delta, & k \in \Z_+.
    \end{matrix}}
\end{equation*}

\noindent The question that motivated the present paper is how to prove an analogous extension theorem replacing $\mathbb{R}^m$ by a maximally real submanifold of $\mathbb{C}^m$. So let $\Sigma \subset \mathbb{C}^m$ be a smooth maximally real submanifold passing through the origin, \textit{i.e}. $\Sigma$ is a real, smooth submanifold of $\mathbb{C}^m$ such that the differential of $z_1|_\Sigma, \dots, z_m|_\Sigma$ are $\mathbb{C}$-linear independent on $\Sigma$. This implies that there exists a family of pair-wise commuting, smooth, $\mathbb{C}$-linear independent, complex vector fields $\{ X_1, \dots, X_m \}$ satisfying $\mathrm{d}(z_j|_\Sigma) X_k = \delta_{jk}$. Since $\Sigma$ is a smooth manifold, we have distributions and smooth functions defined on it, but using the complex structure of the ambient space one can define "real-analytic" functions on $\Sigma$ (the actual name is hypo-analytic) as the restriction of holomorphic functions to $\Sigma$. One interesting aspect of this "real-analytic" (hypo-analytic) functions is that they satisfy some sort of Cauchy estimates for the $X_j$s, loosely speaking, if $u$ is a hypo-analytic function then

\begin{equation*}
    |X^\alpha u| \leq C^{|\alpha|+1} \alpha!,
\end{equation*}

\noindent for every $\alpha$. Actually the reciprocal is also true. So the same picture on $\mathbb{R}^m$ regarding real-analytic functions is valid on maximally real submanifolds of $\mathbb{C}^m$. Now since we have "real-analytic" functions on $\Sigma$, we can define via the estimates on the iterates of the $X_j$s the regular Denjoy-Carleman classes, and so one could ask if the same result proved by Dyn'kin is valid on $\Sigma$. The existence of extensions is valid with the same hypothesis, but for the reverse we needed an extra condition on the regular sequence, namely, the moderate growth condition. The reason why we needed this extra condition lies on the technique that we used. To prove the existence of extensions we were able to adapt Dyn'kin's proof, but to prove the other direction we had to employ the Fourier-Bros-Iagolnitzer transform, F.B.I. transform for short. 
In Dyn'kin's proof he uses the one dimension case to prove the multidimensional one, which we could not do on maximally real submanifolds of $\mathbb{C}^m$. 

In 1983 M. S. Baouendi, C. H. Chang and F. Treves \cite{baouendi1983microlocal} introduced the following F.B.I. transform on maximally-real submanifolds of $\mathbb{C}^m$: If $u \in \mathcal{E}^\prime(\Sigma)$ then for all $(z,\zeta) \in \mathbb{C}^m \times \mathcal{C}_1$,

\begin{equation*}
    \mathfrak{F}[u](z,\zeta) = \left \langle u(z^\prime), e^{i \zeta \cdot (z - z^\prime) - \langle \zeta \rangle \langle z - z^\prime \rangle^2} \Delta(z - z^\prime, \zeta)\right \rangle_{\mathcal{D}^\prime(\Sigma)},
\end{equation*}

\noindent where $\mathcal{C}_1 = \{ \zeta \in \mathbb{C}^m \; : \; |\Im \zeta | < |\Re \zeta| \}$, and $\Delta(z,\zeta)$ is the Jacobian of the map $\zeta \mapsto \zeta + i\langle \zeta \rangle z$. With this F.B.I. transform we were able to prove the following equivalence for smooth functions $u$:

\begin{enumerate}
    \item There exists $V_0 \subset \Sigma$ a neighborhood of the origin such that $u|_{V_0}\in\Cl^\mathcal{M}(V_0; X)$;
    \item There exist $\mathcal{O} \subset \mathbb{C}^m$ a neighborhood of the origin, a function $F \in \mathcal{C}^\infty(\mathcal{O})$, and a constant $C>0$ such that
    \begin{equation}
        \begin{dcases}
            F|_{\mathcal{O}\cap \Sigma} = u|_{\mathcal{O}\cap \Sigma}, \\
            \big|\bar{\partial}_z F(z)\big| \leq C^{k+1}m_k \mathrm{dist}(z,\Sigma)^k,\quad \forall k \in \mathbb{Z}_+ \forall z \in \mathcal{O}.
        \end{dcases}
    \end{equation}
    \item For every $\chi \in \Cl_c^\infty(\Sigma)$, with $0 \leq \chi \leq 1$ and $\chi \equiv 1$ in some open neighborhood of the origin, there exist $V \subset \Sigma$ a neighborhood of the origin and a constant $C > 0$ such that
    \begin{equation} \label{eq:FBI-decay}
        |\mathfrak{F}[\chi u](z,\zeta)| \leq 
        \frac{C^k m_k}{|\zeta|^k}, \qquad \forall (z,\zeta) \in \R\T^\prime_{V} \setminus 0,
    \end{equation}
\end{enumerate}

\noindent where $\mathbb{R} \mathrm{T}^\prime_\Sigma$ is the so-called real-structure bundle of $\Sigma$ (see the beginning of section \ref{sec:Denjoy-Carleman_vectors_maximally_real}). Actually we prove a microlocal version of this result.

As an application of the extension technique we were able to generalize the result obtained by the authors in \cite{braundasilva} to systems of first order non-linear partial differential equations. If $\Omega \subset \R^d \times \R^n$ is an open neighborhood of the origin, and $u \in \mathcal{C}^2(\Omega)$ is a solution of the nonlinear PDE
\begin{equation*}
    \dfrac{\del u}{\del t_j} = f_j\big(x,t,u,u_x\big),
    \eqno{ 1 \leq j \leq n,}
\end{equation*} 
where each $f_j(x,t,\zeta_0, \zeta)$ is a function of class $\Cm$ with respect to $(x,t)$ and holomorphic with respect to $(\zeta_0,\zeta) \in \C \times \C^d$, then
\begin{equation*} 
    \WF_\M(u) \subset \text{Char} (L_1^u, \dots, L_n^u),
\end{equation*}
where the $L_j^u$s are given by
\begin{equation*} 
    L_j^u = \dfrac{\del}{\del t_j} - \sum_{k=1}^d \dfrac{\del f_j}{\del \zeta_k}\big(x,t,u(x,t), u_x(x,t)\big)
    \dfrac{\del}{\del x_k},
    \eqno{ 1 \leq j \leq n.}
\end{equation*}

\noindent This result was already proved by R.F. Barostichi and G. Petronilho \cite{petronilho:11} for the Gevrey classes, and by Z. Adwan and G. Hoepfner \cite{hoepfner:15} for strongly non-quasianalytic Denjoy-Carleman classes.

This paper is organized as follows: In section \ref{sec:Denjoy-Carleman_classes} we recall the basic definitions and properties of the regular Denjoy-Carleman classes, and we prove an extension theorem adapting Dyn'kin's ideas. In section \ref{sec:nonlinear} we prove a microlocal regularity result for solutions of systems of first order non-linear differential equations. Then in section \ref{sec:Denjoy-Carleman_vectors_maximally_real} we study define the Denjoy-Carleman classes on maximally real submanifolds of $\mathbb{C}^m$ and we prove the extension characterization of such classes, and give also a microlocal version of such characterization.

\section{Denjoy-Carleman classes}\label{sec:Denjoy-Carleman_classes}

In this section we recall the basic definitions and properties of regular Denjoy-Carleman classes as defined by Dyn'kin \cite{dyn'kin}, and we shall also prove the first main result of this paper. Following the ideas in \cite{braundasilva} we prove the existence of approximate solutions for a class of systems of complex vector fields.

\subsection{Definitions and basic properties}

\begin{defn} \label{def:regular_sequence}
A sequence $\M = (M_k)_{k=0}^\infty$ of non-negative numbers is \textit{regular} if the the following conditions are satisfied for $m_k = M_k/k!$, $k \in \Z_+$:
\begin{enumerate}
\item[a)] $m_0 = m_1 = 1$;
\item[b)] $m_k^2 \leq m_{k-1}m_{k+1}, \quad k \geq 1$;
\item[c)] $\sup \big( m_{k+1}/m_k \big)^{1/k} < \infty$;
\item[d)] $\lim \, m_k^{1/k} = \infty$.
\end{enumerate}
For a given sequence $\M$ as above and an open set $U \subset \R^N$, the \textit{regular Denjoy-Carleman class} $\Cm(U)$ is the space of all $\Cl^\infty$-smooth functions $f$ in $U$ such that for every compact set $K \subset U$ there is $C > 0$ such that for every $\alpha \in \Z_+^m$ the following estimate holds
\begin{equation*}
    \sup_{K} \big|\del^\alpha f\big| \leq C^{|\alpha|+1} M_{|\alpha|}.
\end{equation*}
\end{defn}
\noindent We say that the regular sequence $\M$ have \textit{moderate growth} if

\begin{enumerate}
    \item[e)] $\sup \Bigg(
        \dfrac{M_k}{\displaystyle \min_{0 \leq n \leq k} M_n M_{k-n}}
    \Bigg)^{\frac{1}{k+1}} < \infty$.
\end{enumerate}
\begin{rmk}
Given a regular sequence $\M$ we fix a constant
\[
c > \max \Bigg \{ 
    \sup \bigg( \dfrac{m_{k+1}}{m_k} \bigg)^{\frac{1}{k}},
    1
\Bigg \}.
\]
In particular, we have $M_{k} \leq c^{\frac{(n+k-1)(k-n)}{2}} M_n$, $0 \leq n \leq k$.
Thus, if $k-n \leq \kappa \in \Z_+$, then 
\begin{equation} \label{eq:M-k-c-kappa-k-M-n}
    M_k \leq (c^\kappa)^{k}M_n.
\end{equation}
If in addition $\M$ has moderate growth, we choose
\[
c > \max \Bigg \{ 
    \sup \bigg( \dfrac{m_{k+1}}{m_k} \bigg)^{\frac{1}{k}}, \;
    \sup \Bigg( \dfrac{M_k}{\displaystyle \min_{0 \leq n \leq k} M_n M_{k-n}} \Bigg)^{\frac{1}{k+1}},\;
    1
\Bigg \}.
\]
In particular, we have
\begin{equation} \label{eq:M-k-moderate-growth-M-n}
    M_k \leq c^{k+1} M_{k-n} M_n,
\end{equation}
for $0 \leq n \leq k$.
\end{rmk}
\begin{defn}
For each $r>0$ we set
\begin{align*}
    h(r)   &= \inf \big\{ 
        m_k r^k : k \in \Z_+ 
    \big\}, \\
    h_1(r) &= \inf \big\{ 
        m_k r^{k-1} : k \in \Z_+ \setminus \{0\} 
    \big\}, \\
    \N(r)  &= \min \big\{ n \in \Z_+ \setminus \{0\} \,:\, h_1(r) = m_n r^{n-1} \big\}.
\end{align*}
\end{defn}
\begin{rmk}
For $0 < r < 1$ one can readily see that
$h(r) \leq h_1(r) \leq h(cr)$ and for $r \geq 1$ we have $h(r) = h_1(r) = 1$.
Moreover one has
\begin{equation} \label{eq:h_1-r_j}
    \dfrac{h_1(r)}{r^j} \leq c^{\frac{j(j+1)}{2}} h_1(c^j r),
\end{equation}
for all $j \in \Z_+$ and $r > 0$.
The function $\N$ is a decreasing step function such that $\N(r)=0$ for every $r \geq 1$ and $\lim_{r \to 0} \N(r) = \infty$.
\end{rmk}
We recall the main property of function $\N$ in the following lemma (see Lemma 2.13 in \cite{braundasilva}).
\begin{lemma} \label{lem:mkrk}
Let $r > 0$. 
If $n \leq k \leq \N(r)$, then
$m_k r^k \leq m_n r^n$.
\end{lemma}
\begin{defn}
Let $U \subset \R^m$ be an open set and let $X = (X_j : 1 \leq j \leq n )$ be an ordered set of commuting $\Cl^\infty$-smooth complex vector fields in $U$.
An \textit{ultra-differentiable vector in} $U$ \textit{with respect to} $\M$ \textit{and} $X$ is a function $f \in \Cl^\infty(U)$ with the following property: for every compact set $K \subset U$ there is $C > 0$ such that for every $\alpha \in \Z_+^n$ the following estimate holds
\begin{equation} \label{eq:def-ultra-vec}
    \sup_{K} \big|X^\alpha f\big| \leq C^{|\alpha|+1} M_{|\alpha|},
\end{equation}
where $X^\alpha = X_1^{\alpha_1} \circ \cdots \circ X_n^{\alpha_n}$ for each $\alpha = (\alpha_1, \dots, \alpha_n) \in \Z_+^n$.
The space of all ultra-differentiable vectors in $U$ with respect to $\M$ and $X$ is denoted by $\Cm(U,X)$.
\end{defn}

\subsection{Extension theorems}

\begin{thm} \label{thm:ext}
Let $U \subset \R^d$ be an open neighbourhood of the origin and 
let $X = (X_j : 1 \leq j \leq m )$ be an ordered set of commuting $\Cl^\infty$-smooth complex vector fields in $U$.
Define a set $\{L_j : 1 \leq j \leq m\}$ of $\Cl^\infty$-smooth complex vector fields in $U \times \R^m$ by
\[
L_j = \frac{\del}{\del v_j} - X_j,
\]
where $v_j$ is the $j^{\mathrm{th}}$-coordinate component of the second factor in $U \times \R^m$.
If $f \in \Cm(U,X)$, then for every $V \Subset U$ there exists $\delta, Q > 0$ and $F \in \Cl (V \times \R^m) \cap \Cl^\infty \big(V \times (\R^m \setminus\{0\})\big)$ such that
\begin{equation*}
    \begin{dcases}
        F(u,0) = f(u), \\
        \big|L_j F(u,v)\big| \leq Q^{k+1} m_k |v|^k,
    \end{dcases}
    \eqno{
    \begin{matrix}
        u \in V,\phantom{_+} &0 < |v| < \delta, \\
        k \in \Z_+, &1 \leq j \leq m.
    \end{matrix}
    }
\end{equation*}
Furthermore, for every $\kappa \in \Z_+$ we can choose $F$ such that for every $\gamma \in \Z_+^m$ with $|\gamma| \leq \kappa$ the function $v \mapsto X^\gamma F(u,v)$ is $\Cl^\infty$-smooth for every $u \in V$.
If $\M$ has moderate growth then $v \mapsto X^\gamma F(u,v)$ is $\Cl^\infty$-smooth for every $u \in V$ and every $\gamma \in \Z_+^m$.
\end{thm}
\begin{proof}
Let us assume that $\M$ has moderate growth.
We follow the steps of section 3 in \cite{nachrichten}. 
The power series 
\begin{equation*}
    F^\sharp(u,v) = 
    \sum_{\alpha \in \Z_+^m}
    f_\alpha(u) v^\alpha = 
    \sum_{\alpha \in \Z_+^m}
    \frac{X^\alpha f(u)}{\alpha!} v^\alpha
\end{equation*}
is the formal solution for initial value the problem
\begin{equation*}
    \begin{dcases}
        F^\sharp(u,0) = f(u),\\
        L_j F^\sharp = 0,
    \end{dcases}
    \eqno{1 \leq j \leq m.}
\end{equation*}
Let $0 < \varepsilon < 1$ be given and denote by $B_\varepsilon \subset \C^m$ the open ball of radius $\varepsilon$ centered at the origin. 
Let $\psi \in \Cl^\infty_c(B_\varepsilon)$ be a real-valued cutoff function such that $\psi \geq 0$, $\psi(z) = \psi(|z|)$ for all $z$, and
\[
    \int_{\C^m} \psi(z) \, \d \lambda(z) = 1,
\]
where $\lambda$ is the Lebesgue measure in $\C^m$. In view of the radial symmetry of $\psi$, we have:
\begin{equation} \label{eq:polynomial}
    \dfrac{1}{|v|^{2m}}
    \int_{\C^m} \psi \bigg( 
        \dfrac{z - v}{|v|} 
    \bigg) P(z) \, \d \lambda(z) = P(v),
\end{equation}
for every polynomial $P(z)$.
Fix $V \Subset U$ an open neighborhood of the origin.
There is constant $C > 0$ such that estimate \eqref{eq:def-ultra-vec} with $K = \overline{V}$ holds for every $\alpha \in \Z_+^m$.
Set 
$\delta =  \big(2(1+\varepsilon)^2cemC \big)^{-1}$. 
For $(u,v) \in V \times \big( \C^m \setminus{\{0\}} \big)$ set
\begin{equation} \label{eq:F}
    F(u,v) = \dfrac{1}{|v|^{2m}} \int_{\C^m}
    \psi\bigg(\dfrac{z-v}{|v|}\bigg)
    \sum_{|\alpha| = 0}^{\N((1+\varepsilon)cemC|z|)}
    f_{\alpha}(u) z^\alpha \; \d \lambda(z).
\end{equation}
Since the integral occurs only in 
$\Omega_\epsilon = \{z \in \C^m : (1-\varepsilon)|v| < |z| < (1+\varepsilon)|v|\}$, 
equation \eqref{eq:F} defines a $\Cl^\infty$-smooth function in $V \times \big( \C^m \setminus{\{0\}} \big)$.
For $0 < \rho < |v| < 2\rho < \delta$, we fix 
$n = \N \big(2(1+\varepsilon)^2cemC \rho \big) - 1$ and evaluate
\begin{align} \label{eq:LjF}
L_j F(u,v) &= 
-\sum_{|\alpha|=n} \frac{X^{\alpha+e_j}f(u)}{\alpha!} v^\alpha + 
\int_{\C^m} \frac{\del}{\del v_j} \bigg[
    \frac{1}{|v|^{2m}}\psi\bigg(\dfrac{z-v}{|v|}\bigg) 
\bigg]
\sum_{|\alpha|=n+1}^{\N((1+\varepsilon)cemC|z|)}
f_\alpha(u) z^\alpha \d \lambda(z) \\
&\quad\, + 
\dfrac{1}{|v|^{2m}} \int_{\C^m}
\psi\bigg(\dfrac{z-v}{|v|}\bigg)
\sum_{|\alpha|=n+1}^{\N((1+\varepsilon)cemC|z|)}
(-X_j)f_{\alpha}(u) z^\alpha \d \lambda(z), \nonumber
\end{align}
for $n+1 \leq \N((1+\varepsilon)cemC|z|)$ in $\Omega_\epsilon$. 
Then we estimate the main factor in the third term of equation \eqref{eq:LjF}:
\begin{align}
    \sum_{|\alpha|=n+1}^{\N((1+\varepsilon)cemC|z|)}
        |X_j f_\alpha(u)||z|^{|\alpha|} 
    &\leq 
    \sum_{|\alpha|=n+1}^{\N((1+\varepsilon)cemC|z|)}
    \dfrac{C^{|\alpha|+2} M_{|\alpha|+1}}{\alpha!}|z|^{|\alpha|} \nonumber \\
    &\leq  
    \sum_{k=n+1}^{\N((1+\varepsilon)cemC|z|)}
    {m+k-1 \choose m-1}
    \dfrac{m^k C^{k+2} M_{k+1}}{k!}|z|^{k} \nonumber \\
    &\leq  
    \sum_{k=n+1}^{\N((1+\varepsilon)cemC|z|)}
    (k+1)^m
    \dfrac{m^k C^{k+2} M_{k+1}}{(k+1)!} |z|^{k} \nonumber \\
    &\leq 
    \sum_{k=n+1}^{\N((1+\varepsilon)cemC|z|)}
    m! \, e^{k+1}
    \dfrac{c^{k} m^k C^{k+2} M_k}{k!} |z|^{k} \nonumber \\
    &= 
    m! \, e \,C^2 \sum_{k=n+1}^{\N((1+\varepsilon)cemC|z|)}
    \dfrac{M_k}{k!} \big(
        (1+\varepsilon)cemC|z|
    \big)^k
    \dfrac{1}{(1+\varepsilon)^k} \nonumber \\
    &\leq 
    m! \, e \, C^2 
    \dfrac{M_{n+1}}{(n+1)!}\big((1+\varepsilon)cemC|z|\big)^{n+1}
    \sum_{k=n+1}^{\infty}
    \dfrac{1}{(1+\varepsilon)^k} \label{ineq:using-lemma-nachrichten} \\
    &\leq
    \dfrac{m! \, e \, C^2 (1+\varepsilon)}{\epsilon} \cdot
    \dfrac{M_{n+1}}{(n+1)!}\big(2(1+\varepsilon)^2cemC \rho \big)^n \nonumber \\
    &=
    \dfrac{m! \, e \, C^2 (1+\varepsilon)}{\epsilon} \;
    h_1 \big(2(1+\varepsilon)^2cemC \rho \big) \label{ineq:using-def-n-N} \\
    &\leq
    \dfrac{m! \, e \, C^2 (1+\varepsilon)}{\epsilon} \;
    h_1 \big(2(1+\varepsilon)^2cemC |v| \big) \label{eq:Lj-3}
\end{align}
where in \eqref{ineq:using-lemma-nachrichten} we applied Lemma \ref{lem:mkrk} and \eqref{ineq:using-def-n-N} follows from the definition of function $\N$ and the choice of $n$.
Analogously, we estimate
\begin{equation} \label{eq:Lj-1}
    \left|
        \sum_{|\alpha|=n}
        \frac{X^{\alpha+e_j}f(u)}{\alpha!} v^\alpha
    \right| 
    \leq
    m! \, e \, C^2 \, h_1
    \big(
        2(1+\varepsilon)^2cemC |v|
    \big)
\end{equation}
and
\begin{equation} \label{eq:Lj-2}
    \sum_{|\alpha|=n+1}^{\N((1+\varepsilon)cemC|z|)}
        |f_\alpha(u)||z|^{|\alpha|}
    \leq
    \dfrac{(m-1)! \, e \, C \, (1 + \varepsilon)c}{(1 + \varepsilon)c - 1} \;
    h_1 \big(2(1+\varepsilon)^2cemC |v| \big).
\end{equation}
Let $g_\beta(z,v) = \del_v^\beta \big\{ \psi((z-v)/|v|)/|v|^{2m} \big\}$. 
We have
\[
    g_\beta(z,v) = \dfrac{1}{|v|^{2m+|\beta|}} \sum_{k=0}^{N_\beta}
    \varphi_k \bigg(
        \dfrac{z-v}{|v|}
    \bigg) P_k\bigg(\dfrac{v}{|v|}\bigg),
\]
where each $\varphi_k$ is supported in $B_\varepsilon$ and each $P_k$ is a polynomial, thus
\begin{equation} \label{eq:g-beta}
    |g_\beta(z, v)| \leq \dfrac{C_{\beta}}{|v|^{2m+|\beta|}}, 
\end{equation}
for some constant $C_{\beta}$ that depends only on $\beta$.
Combining the estimates \eqref{eq:h_1-r_j}, \eqref{eq:Lj-3}, \eqref{eq:Lj-1}, \eqref{eq:Lj-2}, \eqref{eq:g-beta} and the identity \eqref{eq:LjF} we conclude
\begin{equation*}
\big|L_j F(u,v)\big| \leq Q^{k+1}m_k|v|^k,
\eqno{
    \begin{matrix}
        u \in V,\phantom{_+} & 0< |v| < \delta, \\
        k \in \Z_+, &1 \leq j \leq m,
    \end{matrix}
}
\end{equation*}
where $Q > 0$ is a constant obtained combining the former constants.
It remains to prove that $F$ is continuous up to $v=0$, indeed in the following we get an even stronger property.
Let $\gamma, \beta \in \Z_+^m$ be fixed multi-indexes, 
the identity
\begin{align*}
    \del_v^\beta X^\gamma F(u,v) - \beta! X^\gamma f_\beta (u)
    &= 
    \del_v^\beta \bigg\{
        X^\gamma F(u,v) - \sum_{|\alpha| \leq |\beta|} X^\gamma f_\alpha(u) v^\alpha 
    \bigg\}(u,v) \\
    &=
    \int_{\C^m} \del_{v}^{\beta} \bigg\{ 
        \frac{1}{|v|^{2m}} \psi \bigg(\frac{z-v}{|v|}\bigg) 
    \bigg\}
    \sum_{|\alpha|=|\beta|+1}^{\N((1+\varepsilon)cemC|z|)}
    X^\gamma f_{\alpha}(u)z^\alpha \, \d \lambda(z) \\
    &=
    \int_{\C^m} 
    |v|^{2m}g_\beta(|v|w+v,v)
    \sum_{|\alpha|=|\beta|+1}^{\N((1+\varepsilon)cemC ||v|w + v |)} X^\gamma f_{\alpha}(u)
    (|v|w + v)^\alpha
    \, \d \lambda(w),
\end{align*}
implies
\begin{align} 
    \big| 
        \del_v^\beta X^\gamma &F(u,v) - \beta! X^\gamma f_\beta (u)
    \big|
    \leq \nonumber \\
    &\leq 
    \frac{C_\beta}{|v|^{|\beta|}}\int_{|w|\leq\epsilon} \sum_{|\alpha|=|\beta|+1}^{\N((1+\varepsilon)cemC||v|w+v|)}
    C^{1+|\alpha|+|\gamma|} \frac{M_{|\alpha| + |\gamma| }}{\alpha!} ||v|w+v|^{|\alpha|} \d \lambda(w) \nonumber \\
    &\leq 
    \frac{C_\beta}{|v|^{|\beta|}} \int_{|w|\leq\epsilon} \sum_{|\alpha| = |\beta|+1}^{\N((1+\varepsilon)cemC||v|w+v|)}
    C^{1+|\alpha|+|\gamma|} c^{|\alpha|+|\gamma|+1}  \frac{M_{|\gamma|}M_{|\alpha|}}{\alpha!}||v|w+v|^{|\alpha|} \d \lambda(w) \label{eq:M-alpha-gamma-M-alpha-M-gamma} \\ 
    &\leq 
    \frac{C_\beta}{|v|^{|\beta|}} (m-1)! \, e \, C^{1+|\gamma|} c^{1+|\gamma|} M_{|\gamma|}
    \int_{|w|\leq\epsilon} 
    \sum_{k = |\beta|+1}^{\N((1+\varepsilon)cemC||v|w+v|)}
    \frac{M_{k}}{k!} \big((1+\varepsilon)cemC||v|w+v|\big)^{k} \dfrac{1}{(1+\varepsilon)^k} \, \d \lambda(w) \nonumber \\
    &\leq 
    \frac{C_\beta}{|v|^{|\beta|}} (m-1)! \, e \, C^{1+|\gamma|} c^{1+|\gamma|} M_{|\gamma|}
    \int_{|w|\leq\epsilon} 
    \frac{M_{|\beta|+1}}{(|\beta|+1)!} \big((1+\varepsilon)cemC||v|w+v|\big)^{|\beta|+1}
    \sum_{k = |\beta|+1}^{\infty}
    \dfrac{1}{(1+\varepsilon)^k} \, \d \lambda(w) \label{eq:M-k-M-beta-mais-1} \\
    &\leq 
    \Bigg\{
        C_\beta (m-1)! \, e \, C^{1+|\gamma|} c^{1+|\gamma|} M_{|\gamma|}
        \frac{M_{|\beta|+1}}{(|\beta|+1)!}
        \big((1+\varepsilon)cemC \big)^{|\beta|+1}
        \sum_{k = |\beta|+1}^{\infty}
        \dfrac{1}{(1+\varepsilon)^k}
        \int_{|w|\leq\epsilon} 
        (|w|+1)^{|\beta|+1} \, \d \lambda(w) 
    \Bigg\} |v|, \label{eq:del-v-X-gamma}
\end{align}
where the estimate \eqref{eq:M-alpha-gamma-M-alpha-M-gamma} follows from the moderate growth condition (see \eqref{eq:M-k-moderate-growth-M-n}) and in \eqref{eq:M-k-M-beta-mais-1} we have applied Lemma \ref{lem:mkrk}.
Thus, for every $\gamma$ the function $X^\gamma F$ is $\Cl^\infty$-smooth in the $v$-variable up to $v=0$.
In particular, setting $\gamma = \beta = 0$, we can define $F(u,0) = f(u)$ for every $u \in V$ and $F$ is a continuous extension of $f$.
Without the moderate growth assumption, the entire proof follows the same steps above with the following adjustments: for a fixed $\kappa$ one should replace the occurrences of $c$ in $\delta$, $F$ and $n$ by $c^\kappa$, and the factor in the estimate \eqref{eq:M-alpha-gamma-M-alpha-M-gamma} becomes $c^{\kappa (|\alpha|+|\gamma|)}M_{|\alpha|}$.
\color{black}
\end{proof}

\begin{rmk}\label{rmk:holomorphic-parameters}
We note that if $f$ depends holomorphically on some variable $\zeta$, then the extension $F$ would also be holomorphic with respect to $\zeta$.
\end{rmk}

\begin{cor} \label{cor:extension-maximally-real}
If $U \subset \R^m$ is an open neighbourhood of the origin and 
$X = (X_j : 1 \leq j \leq m )$
is an ordered set of commuting linearly independent $\Cl^\infty$-smooth complex vector fields in $U$, and $\kappa \in \Z_+$ then the extension $F$ for $f \in \Cl^\M(U,X)$ in Theorem \ref{thm:ext} is $\Cl^\kappa$-smooth.
When $\M$ has moderate growth, the extension $F$ is $\Cl^\infty$-smooth.
\end{cor}

\begin{proof}
In view of estimate \eqref{eq:del-v-X-gamma} the fact that $X$ is a $\Cl^\infty$-smooth local frame for the complexified tangent bundle $\C \T \R^m$ together with the Mean Value Theorem ensures that $F$ is a $\Cl^\infty$-smooth extension of $f$ when $\M$ has moderate growth and $\Cl^\kappa$-smooth in the general case.
\end{proof}

\begin{cor} \label{cor:extension-nonlinear}
If $U \subset \R^d$ is an open neighbourhood of the origin, $X = (X_j : 1 \leq j \leq m )$ is an ordered set of commuting $\Cl^\M$-smooth complex vector fields in $U$, and $\kappa \in \Z_+$, then for any $f \in \Cl^\M(U)$ the extension $F$ in Theorem \ref{thm:ext} is $\Cl^\kappa$-smooth.
When $\M$ has moderate growth, the extension $F$ is $\Cl^\infty$-smooth.
\end{cor}

\begin{proof}
In order to prove the smoothness up to $v=0$, we must show that for every $\beta,\gamma \in \mathbb{Z}_+^d$ (with $|\gamma| \leq \kappa$ in the geneal case) there is $C_\sharp > 0$ such that the following estimate holds

\begin{equation*}
    |\partial_v^\beta \partial_u^\gamma F(u,v) - \beta! \partial_u^\gamma f_\beta(u)| \leq C_\sharp |v|,
    \eqno{0 < |v| < \delta}.
\end{equation*}
Arguing as in the inequalities \eqref{eq:del-v-X-gamma}  it suffices to show that there exists some positive constant $C>0$ such that for every $u \in V$ and every $\gamma,\alpha \in \mathbb{Z}_+^d$ the following estimate holds true:

\begin{equation}
    \left| \partial_u^\gamma f_\alpha (u) \right| \leq C^{|\gamma|+|\alpha|+1} \frac{M_{|\gamma|+|\alpha|}}{\alpha!}.
\end{equation}
This estimate is proven in \cite{hoepfner:10} (see estimate 2.6 on page 1724) and \cite{petronilho:09} (see Lemma 3.2), and it relies on the log-convexity for the sequence $m_k = M_k/k!$.
\end{proof}

\section{Systems of first-order nonlinear PDEs}\label{sec:nonlinear}

In this section we generalize the main result in \cite{braundasilva} for systems of first-order nonlinear PDEs. Let $\M$ be a regular sequence and let us denote the coordinates in $\R^d \times \R^n$ by $(x,t) = (x_1, \dots, x_d, t_1, \dots, t_n)$.

\begin{defn}
As in \cite{bch}, we define the (usual) F.B.I. transform of a compactly supported distribution $u$ by
\begin{equation*}
\F[u](x,\xi)=u_y\left(e^{i(x-y)\cdot\xi-|\xi|(x-y)^2}\right).
\end{equation*}
\end{defn}
In \cite{furdos}, it is proved that a distribution $u$ on $\Omega$ belongs to $\Cm(\Omega)$ if and only if for every $x_0\in\Omega$ there are $\chi\in\Ci_c(\Omega)$, with $\chi\equiv 1$ in an open neighborhood of $x_0$, $U\subset\Omega$ an open neighborhood of $x_0$ and a positive constant $A$ such that:
\begin{equation}\label{eq:FBI-decay-Cm}
|\F[\chi u](x,\xi)|\leq \frac{A^{k+1}M_k}{|\xi|^k}, \quad k \in \Z_+, \; x \in U, \; \xi \in \R^N \setminus \{0\}.
\end{equation}
This last inequality can be used to microlocalize the notion of $\Cm$-regularity.
As usual, a subset $\Gamma \subset \R^N$ is said to be a \textit{cone} if for every $x \in \Gamma$ and every $t > 0$ we have $tx \in \Gamma$.
\begin{defn}
Let $u$ be distribution $u$ on $\Omega$ and fix $(x_0,\xi_0) \in \Omega \times \R^N$, $\xi_0 \neq 0$. 
\begin{enumerate}
    \item We say that $u$ is $\Cm$-\textit{regular at} $(x_0, \xi_0)$ if there are $\chi\in\Ci_c(\Omega)$, with $\chi\equiv 1$ in an open neighborhood of $x_0$, $U \Subset \Omega$ an open neighborhood of $x_0$ and $\Gamma \subset \R^N \setminus \{0\}$ an open cone, with $\xi_0 \in \Gamma$, such that
    \begin{equation*}
        |\F[\chi u](x,\xi)| \leq \frac{A^{k+1}M_k}{|\xi|^k}, \quad k \in \Z_+, \; x \in U, \; \xi \in \Gamma.
    \end{equation*}
    \item The \textit{Denjoy-Carleman wave-front set of} $u$ \textit{with respect to} $\M$ \textit{at} $x_0$ is given by
    \[
    \WF_\M(u)|_{x_0} \doteq \{(x_0, \xi) \, : \, u \textrm{ is } \textit{not} \;\; \Cm \textrm{-regular at } (x_0, \xi)\}.
    \]
\end{enumerate}
\end{defn}

\comment{
\begin{thm}
Let 
$\Omega \subset \R^d \times \R^n$ 
be an open neighborhood of the origin.
Let $u \in \Cl^2(\Omega)$ be a solution of the system of nonlinear PDEs:
\begin{equation*}
    \dfrac{\del u}{\del t_j} = f_j(x,u,u_x),
    \eqno{ 1 \leq j \leq n,}
\end{equation*} 
where each $f_j(x, \zeta_0, \zeta)$ is a $\Cl^\infty$-function with $\Cm$-regularity with respect to the $x$-variable and holomorphic with respect to $(\zeta_0,\zeta) \in \C \times \C^d$.
Then:
\begin{equation}\label{eq:inclusao}
    \left.\WF_\M(v_0)\right|_0\subset \big\{(0,\xi) \in \Omega^\prime \times \R^N \,:\,\Im \, b(0) \cdot \xi \geq 0 \big\},
\end{equation}
where $v_0 \in \mathcal{C}^2(\Omega^\prime)$ is given by $v_0(x) = v(x,0)$, $x \in \Omega^\prime$, and $b(x) = \nabla_\zeta g(x, v_0(x), v_{0x}(x))$.
\end{thm}
}

\begin{lemma}\label{lem:FBI}
Let $\Omega \subset \R^{N}$ be an open neighborhood of the origin.
Let
\begin{equation*}
    L_j = \dfrac{\del}{\del r_j} + \sum_{\ell=1}^N a_{j \ell}(x)\dfrac{\del}{\del x_\ell},
    \eqno{1 \leq j \leq n,}
\end{equation*}be a vector field in $\Omega \times \R^n$ where $a_{j \ell} \in \mathcal{C}^1(\Omega)$, $1 \leq \ell \leq N$, $1 \leq j \leq n$.
Suppose that for each $1 \leq \ell \leq N$ there exists $Z_\ell \in \mathcal{C}^1(\Omega \times \R^n)$ and $Q, \delta > 0$ such that
\begin{equation*}
    \begin{dcases}
        Z_\ell(x,0) = x_\ell, \\
        \big|L_j Z_\ell(x,r)\big| \leq Q^{k+1} m_k |r|^k, 
    \end{dcases}
    \eqno{
        \begin{array}{l r}
         x \in \Omega, &0 < |r| < \delta, \\
        k \in \Z_+, &1 \leq j \leq n.
        \end{array}
    }
\end{equation*}
Let $\xi_0 \in \R^N \setminus \{0\}$ and $j_0 \in \{ 1, \dots, n\}$ be such that $\Im\,a_{j_0}(0) \cdot \xi_0 < 0$.
Let $\Psi \in \mathcal{C}^1(\Omega \times \R^n)$ be such that
\begin{equation*}
    \big|L_{j_0} \Psi(x,r)\big| \leq Q^{k+1} m_k |r|^k,
    \eqno{
    \begin{matrix*}[r]
        x \in \Omega, & 0 < |r| < \delta, \\
        &k \in \Z_+.
    \end{matrix*}
    }
\end{equation*}
Then there exist an open cone $\Gamma \subset \R^N\setminus \{0\}$, open neighborhoods of the origin $V \Subset U \Subset \Omega$ , a cutoff function $\chi \in \Ci_c(\Omega)$, with $\chi = 1$ on $U$, and a constant $A > 0$ such that $\xi_0 \in \Gamma$ and
\begin{equation*}
    \left|\F\!\left[\chi \Psi_0\right](x,\xi)\right| \leq \dfrac{A^{k+1}M_k}{|\xi|^k}, 
    \eqno{(x,\xi) \in V \times \Gamma, \quad k \in \Z_+,}
\end{equation*}
where $\Psi_0(x) = \Psi(x,0)$.
\end{lemma}

\begin{proof}
We apply Lemma 4.1 of \cite{braundasilva}, see also section 2 of \cite{asano:95}, to the vector field $L^\sharp = L_{j_0}$.
Thus we obtain a vector field $L^\sharp_1$ over an open set $\Omega_1 \subset \Omega$ such that $L^\sharp_1 Z = 0$, $L^\sharp_1 r_j = \delta _{j_0j}$ and
\begin{equation*}
    \big|L^\sharp_1 \Psi(x,r)\big| \leq Q^{k+1} m_k |r|^k,
    \eqno{
    \begin{matrix*}[r]
        x \in \Omega, & 0 < |r| < \delta, \\
        &k \in \Z_+.
    \end{matrix*}
    }
\end{equation*}
We have $\d(H \; \d r_1 \wedge \dots \wedge \widehat{\d r_{j_0}} \wedge \dots \wedge \d r_n \wedge \d Z) = (-1)^{j_0+1}(L_1^\sharp H) \d r \wedge \d Z$, for every function $H \in \mathcal{C}^1(\Omega_1)$.
The rest of the proof is completely analogous to the proof of Lemma 4.2 in \cite{braundasilva}.
\end{proof}

\begin{thm}
Let $\Omega \subset \R^d \times \R^n$ be an open neighborhood of the origin.
Let $u \in \mathcal{C}^2(\Omega)$ be a solution of the nonlinear PDE:
\begin{equation*}
    \dfrac{\del u}{\del t_j} = f_j\big(x,t,u,u_x\big),
    \eqno{ 1 \leq j \leq n,}
\end{equation*} 
where each $f_j(x,t,\zeta_0, \zeta)$ is a function of class $\Cm$ with respect to $(x,t)$ and holomorphic with respect to $(\zeta_0,\zeta) \in \C \times \C^d$.
Then:
\begin{equation} \label{eq:WF_M_u_subset_T0}
    \WF_\M(u) \subset \T^{0}(\V^u),
\end{equation}
where $\V^u$ is the $\Cl^1$-smooth involutive structure defined by the linearized operators:
\begin{equation*} 
    L_j^u = \dfrac{\del}{\del t_j} - \sum_{k=1}^d \dfrac{\del f_j}{\del \zeta_k}\big(x,t,u(x,t), u_x(x,t)\big)
    \dfrac{\del}{\del x_k},
    \eqno{ 1 \leq j \leq n.}
\end{equation*}
\end{thm}

\begin{proof}
In this proof we follow closely the proof of the Theorem $4.1$ of \cite{asano:95}. 
We shall prove the inclusion \eqref{eq:WF_M_u_subset_T0} at the origin.
A direction $(0, \xi, \tau) \in \Omega \times \R^d \times \R^n$ belongs to $\T^0(\mathcal{V}^u)$ if and only if $\tau = a(0) \cdot \xi$, where $a(x,t) = \del_{\zeta}f(x,t,u(x,t),u_x(x,t))$.
This is equivalent to the validity of the identity
\begin{equation} \label{}
    \cos \theta \, \Im \, a_j(0) \cdot \xi + \sin \theta (\tau_j - \Re \, a_j(0) \cdot \xi) = 0
\end{equation}
for all $\theta \in [0,2\pi)$ and for all $1 \leq j \leq n$.
Thus, if $(0, \xi_0, \tau_0) \notin \T^0(\V^u)$, there exist $\theta \in [0, 2\pi)$ and $j_0$ be such that $\cos \theta \, \Im \, a_{j_0}(0) \cdot \xi_0 + \sin \theta (\tau_{j_0} - \Re \, a_{j_0}(0) \cdot \xi_0) > 0$. 
We consider $u$ as a function in $\Cl^2(\Omega \times \R^n)$ that does not depend on the $r$-variable, thus, it is a solution of the following nonlinear PDE:
\begin{equation} \label{eq:vs}
    \dfrac{\del u}{\del r_j} = f_j^\theta(x,t,u,u_x,u_t),
\end{equation}
where $f_j^\theta(x,t,\zeta_0,\zeta,\tau) = e^{-i\theta}(\tau_j - f_j(x,t,\zeta_0,\zeta))$.
Now consider the system of vector fields
\begin{equation*}
    L_j^\theta = \frac{\del}{\del r_j} - 
    \sum_{k=1}^{d} 
    \dfrac{\del f_j^\theta}{\del \zeta_k}(x,t,\zeta_0,\zeta,\tau)
    \frac{\del}{\del x_k} -
    \sum_{\ell=1}^{n} 
    \dfrac{\del f_j^\theta}{\del \tau_\ell}(x,t,\zeta_0,\zeta,\tau)
    \frac{\del}{\del t_\ell},
    \eqno{1 \leq j \leq n},
\end{equation*}
in $\mathrm{Dom} \, f \times \R^n$. 
To finish the proof, it suffices to  apply Lemma \ref{lem:FBI} for the $\Cl^1$-vector fields
\begin{equation*}
    \big(L_j^\theta\big)^u = \frac{\del}{\del r_j} - 
    \sum_{k=1}^{d} 
    \dfrac{\del f_j^\theta}{\del \zeta_k}(x,t,u,u_x,u_t)
    \frac{\del}{\del x_k} -
    \sum_{\ell=1}^{n} 
    \dfrac{\del f_j^\theta}{\del \tau_\ell}(x,t,u,u_x,u_t)
    \frac{\del}{\del t_\ell},
    \eqno{1 \leq j \leq n},
\end{equation*}
and to some $\Cl^1$-approximate solution that extends $u$.
To fulfill the hypothesis of Lemma \ref{lem:FBI}, we shall use Corollary \ref{cor:extension-nonlinear} and the Remark \ref{rmk:holomorphic-parameters} to the holomorphic Hamiltonian vector fields
\begin{equation*}
    H_j^\theta = L_j^\theta + h_{j0}^\theta \frac{\del}{\del \zeta_0}
    + \sum_{k=1}^d h_{jk}^\theta \frac{\del}{\del \zeta_k} 
    + \sum_{\ell=1}^n h_{j(d+\ell)}^\theta \frac{\del}{\del \tau_\ell} 
    \eqno{1 \leq j \leq n},
\end{equation*}
where
\begin{align*}
    h_{j0}^\theta(x,t,\zeta_0,\zeta,\tau) &= f_j^\theta(x,t,\zeta_0,\zeta,\tau) - 
    \sum_{k=1}^d \zeta_k\frac{\del f_j^\theta}{\del \zeta_k}(x,t,\zeta_0,\zeta,\tau) -
    \sum_{\ell=1}^n \tau_\ell \frac{\del f_j^\theta}{\del \tau_\ell}(x,t,\zeta_0,\zeta,\tau),\\
    h_{ji}^\theta(x,t,\zeta_0,\zeta,\tau) &= \begin{dcases}
        \frac{\del f_j^\theta}{\del x_i}(x,t,\zeta_0,\zeta,\tau) + \zeta_i \frac{\del f_j^\theta}{\del \zeta_0}(x,t,\zeta_0,\zeta,\tau), &1 \leq i \leq d, \\
        \frac{\del f_j^\theta}{\del t_i}(x,t,\zeta_0,\zeta,\tau) + \tau_i \frac{\del f_j^\theta}{\del \zeta_0}(x,t,\zeta_0,\zeta,\tau), &d+1 \leq i \leq d + n,
    \end{dcases}
\end{align*}
for each $1 \leq j \leq n$, and for the initial conditions $x, t, \zeta_0$, noticing that these vector fields commute pairwise and that the identity
$(L_j^\theta)^u \Phi^u = \left(H_j^\theta \Phi\right)^u$,
holds for every $\Cl^1$-function $\Phi(x,t,r,\zeta_0,\zeta,\tau)$ that is holomorphic with respect to $(\zeta_0,\zeta,\tau)$.
\end{proof}

\section{Denjoy-Carleman vectors in maximally real manifolds}\label{sec:Denjoy-Carleman_vectors_maximally_real}

Let $\Sigma \subset \mathbb{C}^m$ be a $\Cl^\infty$-smooth submanifold. We say that $\Sigma$ is \textit{maximally real} if for every $p \in \Sigma$ one of the following (equivalent) conditions holds true:

\begin{itemize}
    \item 
        $\C \T_p \C^m \simeq \T_p^{(0,1)} \C^m \oplus \C \T_p \Sigma$;
    
    \item 
        The pullback map $j^\ast : \C \T_p^\ast \C^m \rightarrow \C \T_p^\ast \Sigma$, where $j$ is the inclusion map $\Sigma \hookrightarrow \C^m$, induces an isomorphism $j^\ast_{1,0} : {\T_{(1,0)}}_p \C^m \to \C \T_p^\ast \Sigma$,
    
    \item 
        The one forms $\d (z_1|_\Sigma), \dots, \d (z_m|_\Sigma)$ are linearly independent at $p$,
\end{itemize}

\noindent here we are using the notation $\T_p^{(0,1)} \C^m$ for the $(1,0)$-complex vector fields at $p$, and ${\T_{(1,0)}}_p \C^m$ for the $(1,0)$-forms at $p$. 

The image of $\mathrm{T}^\ast\Sigma$ under the isomorphism $( j_{1,0}^\ast )^{-1}$ is the \textit{real structure bundle} of $\Sigma$ and it is denoted by 
$\mathbb{R} \mathrm{T}^\prime_{\Sigma}$,
it is a real vector bundle over $\Sigma$ whose fiber dimension is equal to $m$. 

After applying a biholomorphism on $\Sigma$, we can assume that on some open neighborhood $\Omega \subset \C^m$ of the origin, the submanifold $\Sigma$ is the graph of a $\Cl^\infty$-smooth map $\varphi : U \to \R^m$, with $\varphi(0) = 0$ and $\mathrm{d} \varphi(0) = 0$, where $U \subset \R^m$ is an open neighbourhood of the origin, ensuring the following local expression
\begin{equation*}
    \Sigma \cap \Omega = 
    \{x+iy \in \Omega : y-\varphi(x) = 0\} = 
    Z(U),
\end{equation*}
where $Z = (Z_1, \dots, Z_m) = (x \mapsto x + i \varphi(x))$. We shall also assume that $|\varphi(x) - \varphi(x^\prime)| \leq C_\varphi |x - x^\prime|$, for all $x,x^\prime \in U$, where the constant $C_\varphi>0$ is as small as we want, keeping in mind that in order to diminish $C_\varphi$ one need to shrink $U$ around the origin. In the following we shall assume $C_\varphi < 1$.

In $\Sigma \cap \Omega$, the real structure bundle can be described as follows: a complex direction $\zeta \in \C^m$ belongs to $\R \T^\prime_\Sigma|_{Z(x)}$ if, and only if, $ \zeta={}^{\mathrm{t}}Z_x(x)^{-1}\xi$, for some $\xi \in \R^m$.

Since our results are local, from now on we fix the open set $\Omega$ and the map $Z(x)$.

\subsection{Almost analytic extension}

Reducing the open set $U$ if necessary we may assume that the matrix
\[
    [Z_x] = 
    \bigg[ 
        \delta_{j\ell} + i\dfrac{\del \varphi_j}{\del x_\ell} 
    \bigg]
\]
is invertible in $U$.
Set for each $1 \leq k \leq m$
\[
    Y_k = \sum_{\ell=1}^m a_{k\ell} \dfrac{\del}{\del x_\ell} + b_{k\ell} \dfrac{\del}{\del y_\ell},
\]
where $\,^t[a_{k \ell}]$ is the inverse of $[Z_x]$ and $b_{k \ell} = \sum_{j=1}^m a_{kj} \del \varphi_\ell/\del x_j$.
Thus $Y = (Y_1, \dots, Y_m)$ is a frame of vector fields in $\Omega$ that are tangent to $\Sigma$ and satisfies the relations
\[
    Y_k z_j|_{\Sigma \cap \Omega} = \delta_{jk},
\]
for all $1 \leq j, k \leq m$.
From now on we denote by $X = (X_1, \dots, X_m)$ the restriction of $Y$ to $\Sigma \cap \Omega$.
\comment{
Thus, setting for each $1 \leq k \leq m$
\[
    X_k = \sum_{\ell=1}^m a_{k\ell} \dfrac{\del}{\del x_\ell} + b_{k\ell} \dfrac{\del}{\del y_\ell},
\]
entails
\[
    \delta_{jk} 
    =   
    X_k(x_j+i\varphi_j(x)) 
    = 
    \sum_{\ell=1}^m \bigg( 
        a_{k\ell} \dfrac{\del}{\del x_\ell} + b_{k\ell} \dfrac{\del}{\del y_\ell} 
    \bigg) (x_j + i \varphi_j(x)) 
    = 
    \sum_{\ell=1}^m a_{k\ell} \bigg( 
        \delta_{j\ell} + 
        i \dfrac{\del \varphi_j}{\del x_\ell}
    \bigg),
\]
for every $1 \leq j, k \leq m$.
Therefore $\,^t[a_{k \ell}]$ is the inverse matrix of
\[
    \bigg[ 
        \delta_{j\ell} + i\dfrac{\del \varphi_j}{\del x_\ell} 
    \bigg]
\]
and $b_{k \ell} = \sum_{j=1}^m a_{kj} \del \varphi_\ell/\del x_j$.
}
\begin{thm} \label{thm:extension-delbar-special-coords}
Let $f \in \Cl^\mathcal{M}(U,X)$ and let $\kappa \in \Z_+$.
Then there exist $\mathcal{O} \subset \mathbb{C}^m$ a neighborhood of the origin, a function $F \in \mathcal{C}^\kappa(\mathcal{O})$, and a constant $C>0$ such that
\begin{equation*}
    \begin{dcases}
        F|_{\mathcal{O}\cap \Sigma} = f|_{\mathcal{O}\cap \Sigma} \\
        \big|\bar{\partial}_z F(z)\big| \leq C^{k+1} m_k \mathrm{dist}(z,\Sigma)^k,
    \end{dcases}
    \eqno{k \in \mathbb{Z}_+, \; z \in \mathcal{O}.}
\end{equation*}
When $\M$ has moderate growth, then the extension F is $\Cl^\infty$-smooth.
\end{thm}
\begin{proof}
After the change of coordinates $(u,v) : \Omega \to \C^m$ given by
\[
    \begin{cases}
        u = x, \\
        v = y - \varphi(x),
    \end{cases}
\]
the vector fields $\del/\del \bar{z}_j$ and $X_k$ are expressed by
\begin{align*}
    \dfrac{\del}{\del \bar{z}_j}
    &=
    \dfrac{1}{2}
    \bigg\{
        \dfrac{\del}{\del u_j} 
        + 
        i \sum_{\ell=1}^m
        \bigg(
            \delta_{j \ell} +
            i \dfrac{\del \varphi_\ell}{\del x_j}
        \bigg) 
        \dfrac{\del}{\del v_\ell}
    \bigg\}, \\
    X_k
    &= 
    \sum_{\ell=1}^m  a_{k\ell} \dfrac{\del}{\del u_\ell}.
\end{align*}
The $\Cl^\infty$-smooth vector fields
$L_k = \del/\del v_k - i X_k$, $1 \leq k \leq m$, satisfy the hypothesis of Theorem \ref{thm:ext} and Corollary \ref{cor:extension-maximally-real}, and they also satisfy
\[
   \frac{\partial}{\partial \bar{z}} = A(u,v) L,
\]
where $A(u,v)$ is the matrix $\big(2/i [a_{k\ell}]\big)^{-1}$.
\comment{
\[
\begin{bmatrix}
L_1 \\
\vdots \\
L_m
\end{bmatrix}
= \dfrac{2}{i}
\begin{bmatrix}
a_{11} & \cdots & a_{1m} \\
\vdots & & \vdots \\
a_{m1} & \cdots & a_{mm}
\end{bmatrix}
\begin{bmatrix}
\del/\del \bar{z}_1 \\
\vdots \\
\del/\del \bar{z}_m
\end{bmatrix}
=
\begin{bmatrix}
\del/\del v_1 \\
\vdots \\
\del/\del v_m
\end{bmatrix}
- i
\begin{bmatrix}
X_1    \\
\vdots \\
X_m
\end{bmatrix}.
\]

Since the matrix $(a_{ij})$ is inverbile in an 
Now let us recapitulate everything that was done here in the following proposition:
}
\end{proof}

\comment{
\begin{prop} \label{prop:special-coordinate-system}
Let $\Sigma \subset \C^m$ be a maximally real submanifold passing through the origin, and let $X = (X_1, \dots, X_m)$ be the dual basis of $\d (z_1|_\Sigma), \dots, \d (z_m|_\Sigma)$ near the origin. 
Then there exists a smooth coordinate system $(u,v)$ around the origin in $\C^m \simeq \R^{2m}$ and a smooth invertible matrix $A(z)$ such that $\Sigma$ near the origin is given by $v=0$, and 
\[
   \frac{\partial}{\partial \bar{z}} = A(z) L,
\]
where $L_j =  \frac{\partial}{\partial v_j} - i X_j$, $j=1, \dots, m$.
\end{prop}
}

\subsection{A microlocal characterization}

In this section we shall relate the three main concepts of this paper: Denjoy-Carleman vectors, almost analytic extensions and the F.B.I. transform. Before doing so let us briefly recall the definition and some properties of the F.B.I. transform on maximally real submanifods of $\C^m$. 

For every $\kappa>0$ we write
\begin{equation*}
    \mathfrak{C}_\kappa\doteq\{\zeta\in\mathbb{C}^m\;:\;|\Im \, \zeta|<\kappa|\Re \, \zeta|\},
\end{equation*}
and if $\zeta\in\mathbb{C}^m$ we write $\langle\zeta\rangle^2\doteq \zeta\cdot\zeta=\zeta_1^2+\cdots+\zeta_m^2$. 
Taking the main branch of the square root we can define $\langle\zeta\rangle\doteq[\langle\zeta\rangle^2]^{1/2}$, for $\zeta\in\mathfrak{C}_1$. 

\begin{defn}
We shall say that the maximally real submanifold $\Sigma$ of $\mathbb{C}^m$ is well positioned at the origin if for every $\lambda >0 $ there are positive numbers $\kappa$ and $\kappa^\prime$, with $0<\kappa<1$, and an open neighborhood $\Omega^\prime$ of the origin on $\Sigma$ such that
\begin{equation}\label{eq:maximally-real-well-positioned}
    \begin{cases}
        |\Im \, \zeta| < \kappa \, |\Re \, \zeta|, \\
        \Re\big\{
            i\zeta\cdot(z-z^\prime)- \lambda \langle\zeta\rangle\langle z-z^\prime\rangle^2
        \big\} \leq - \kappa^\prime |\zeta| |z-z^\prime|^2,
     \end{cases}
\end{equation}
for all $z,z^\prime\in \Omega^\prime$ and all $\zeta \in \big(\mathbb{R}\mathrm{T}^\prime_{\Sigma}|_z\big) \cap \big(\mathbb{R}\mathrm{T}^\prime_{\Sigma}|_{z^\prime}\big)$.
\comment{
  We shall say that $\Sigma$ is very well positioned at the origin if, given any $0<\kappa<1$, there is an open neighborhood $\Omega$ of the origin in $\Sigma$ such that \eqref{eq:maximally-real-well-positioned} is valid.} 
\end{defn}
Without loss of generality we assume that $\Sigma$ is well positioned at the origin  with $\Omega^\prime = \Sigma \cap \Omega$ (see Proposition IX.2.2 of \cite{trevesbook}).
\comment{
We state a result that can be found in Treves' book that insures that after a biholomorphism any maximally real submanifold of $\C^m$ is well-positioned at a given point.

\begin{prop}[IX.$2.2.$ of \cite{treveslivro}]
Given any maximally real submanifold $\Sigma$ of $\mathbb{C}^m$ and any point $p$ of $\Sigma$ there is a biholomorphism $H$ of an open neighborhood $\mathcal{O}$ of $p$ in $\mathbb{C}^m$ onto an open neighborhood of the origin, with $H(p)=0$, such that $H(\Sigma\cap\mathcal{O})$ is very well positioned at $0$.
\end{prop}
}
\comment{
\noindent Now consider the map

\begin{equation}\label{eq:map-deform}
    \mathfrak{C}_\kappa\times \mathbb{C}^m\ni(\zeta,z)\mapsto\zeta+iz\langle\zeta\rangle.
\end{equation}

\noindent Since this map is positive homogeneous in $\zeta$ and it is equal to the identity map in $\mathfrak{C}_\kappa$ when $z=0$, for every $\kappa<\kappa^\prime<1$ there exists an open neighborhood of the origin $\mathcal{O}$ in $\mathbb{C}^m$ such that

\begin{equation*}
    \zeta+iz\langle\zeta\rangle\in\mathfrak{C}_{\kappa^\prime},
\end{equation*}

\noindent for all $\zeta\in\mathfrak{C}_\kappa$ and $z\in\mathcal{O}$. We shall use the notation 

\begin{equation*}
    \Delta(z,\zeta)=\det(\mathrm{Id}+i(z\odot \zeta)/\langle\zeta\rangle),
\end{equation*}

\noindent for $\zeta\in\mathfrak{C}_\kappa\setminus 0$ and $z\in\mathbb{C}^m$, where $(z\odot\zeta)=(z_i\zeta_j)_{i,j=1,\dots,m}$. This can be seen as the Jacobian of the map \eqref{eq:map-deform}. 
}
\begin{defn}
Let $u\in\mathcal{E}^\prime(\Sigma \cap \Omega)$ and $\lambda > 0$. We define the $\lambda$-F.B.I. transform of $u$ (or just F.B.I. transform if $\lambda = 1$) by

\begin{equation*}
    \mathfrak{F}^{\lambda}[u](z,\zeta) \doteq \left\langle 
        u(z^\prime), e^{i\zeta\cdot(z-z^\prime) - \lambda \langle\zeta \rangle\langle z - z^\prime \rangle^2} \Delta(\lambda(z-z^\prime),\zeta)
    \right\rangle_{\mathcal{D}^\prime},
\end{equation*}

\noindent for $z\in\mathbb{C}^m$ and $\zeta\in\mathfrak{C}_1$, where $\Delta(z,\zeta)$ is the Jacobian of the map $\zeta\mapsto\zeta+iz\langle\zeta\rangle$.

\end{defn}

\begin{rmk}
We are using the notation $\langle u,\phi\rangle_{\mathcal{D}^\prime}$ for the duality between distributions and smooth functions, and saving $\langle \cdot \rangle$ for the square root of the euclidean inner product of vectors in $\mathbb{C}^m$. 
\end{rmk}

In view of Theorem \ref{thm:extension-delbar-special-coords} and proceeding analogously as the proof of Theorem 3.4 of \cite{braun:2022} we may state the following theorem.

\begin{thm} \label{thm:vector-extension-FBI}
Let $\mathcal{M}$ be a regular sequence with moderate growth.
Let $\Sigma \subset \C^m$ be a maximally real submanifold passing through the origin, and let $X = (X_1, \dots, X_m)$ be the dual basis of $\d (z_1|_\Sigma), \dots, \d (z_m|_\Sigma)$ near the origin.
Let $u \in \mathcal{D}^\prime(\Sigma)$. 
The following are equivalent:
\begin{enumerate}
    \item There exists $V_0 \subset \Sigma$ a neighborhood of the origin such that $u|_{V_0}\in\Cl^\mathcal{M}(V_0; X)$;
    \item There exist $\mathcal{O} \subset \mathbb{C}^m$ a neighborhood of the origin, a function $F \in \mathcal{C}^\infty(\mathcal{O})$, and a constant $C>0$ such that
    \begin{equation}
        \begin{dcases}
            F|_{\mathcal{O}\cap \Sigma} = u|_{\mathcal{O}\cap \Sigma}, \\
            \big|\bar{\partial}_z F(z)\big| \leq C^{k+1}m_k \mathrm{dist}(z,\Sigma)^k,\quad \forall k \in \mathbb{Z}_+ \forall z \in \mathcal{O}.
        \end{dcases}
    \end{equation}
    \item For every $\chi \in \Cl_c^\infty(\Sigma)$, with $0 \leq \chi \leq 1$ and $\chi \equiv 1$ in some open neighborhood of the origin, there exist $V \subset \Sigma$ a neighborhood of the origin and a constant $C > 0$ such that
    \begin{equation} \label{eq:FBI-decay}
        |\mathfrak{F}[\chi u](z,\zeta)| \leq 
        \frac{C^k m_k}{|\zeta|^k}, \qquad \forall (z,\zeta) \in \R\T^\prime_{V} \setminus 0.
    \end{equation}
\end{enumerate}
\end{thm}

We point out that the moderate growth condition is applied in the estimate of page 16 in the proof of Theorem 3.4 of \cite{braun:2022}.

In the following, we shall prove a microlocal version of Theorem \ref{thm:vector-extension-FBI}. 
To define the concept of Denjoy-Carleman microlocal regularity on the $\Cl^\infty$-smooth manifold $\Sigma$, we follow \cite{cordaro2016hyperfunctions}.
Let $\Gamma \subset \R^m \setminus \{0\}$ be an open cone and let $V \Subset U$ be an open neighbourhood of the origin.
Given $\delta > 0$, the \textit{wedge} with \textit{edge} $Z(V)$, \textit{directrix} $\Gamma$ and \textit{height} $\delta$ is the open set
\begin{equation*}
    \mathcal{W}_\delta (V,\Gamma) = \{ Z(x) + iv \; : \; x \in V, v \in \Gamma, \, \text{and} \; |v| < \delta \}.
\end{equation*}
A function $f \in \Cl^1(\mathcal{W}_\delta (V,\Gamma))$ has \textit{slow growth at the edge} if there are $C > 0$ and $N \in \Z_+$ such that $|f(Z(x)+iv)| \leq C/|v|^N$ for every $x \in V$ and $v \in \Gamma$, $|v| < \delta$.
We also say that $f$ is $\M$-\textit{almost analytic at the edge} if there is $C_f > 0$ such that $\big|\bar{\del}_{z} f(Z(x)+iv) \big| \leq C_f^{k+1}m_k|v|^k$, for every $k \in \Z_+$.
If an $\M$-almost analytic function $f \in \Cl^1(\mathcal{W}_\delta (V,\Gamma))$ has slow growth at the edge, then one can define its distribution boundary value by
\begin{equation*}
    \langle \mathrm{b}_\Gamma(f), \psi \rangle = \lim_{t \to 0^+} 
    \int_V f(Z(x)+it\gamma) \psi(Z(x)) \, \d Z(x),
\end{equation*}
where $\gamma \in \Gamma$ is any fixed direction and $\psi \in \Cl^\infty_c(Z(V))$.
\begin{defn}
Let $u \in \mathcal{D}^\prime(\Sigma \cap \Omega)$ and let $(p_0, \zeta_0) \in \R\T^\prime_{\Sigma} \setminus 0$, so $\zeta_0 = {}^\mathrm{t} \mathrm{d} Z(p_0)^{-1} \xi_0$, with $\xi_0 \in \mathbb{R}^m$. 
We say that $u$ is \textit{microlocally a} $\Cl^\mathcal{M}$-vector with respect to $X$ at $(p_0, \zeta_0)$ if there exist an open set $V \Subset U$, with $p_0 \in Z(V)$, $\delta > 0$, open cones $\Gamma_j$, $1 \leq j \leq \ell$, with $\xi_0 \cdot \Gamma_j < 0$, and $\M$-almost analytic functions $f_j \in \Cl^1(\mathcal{W}_\delta (V,\Gamma_j))$ with slow growth at the edge such that
\begin{align*}
    u|_{Z(V)} = \sum_{j=1}^\ell \mathrm{b}_{\Gamma_j}(f_j).
\end{align*}
The \textit{Denjoy-Carleman wave-front set of} $u$ \textit{with respect to} $X$ is the subset of $\R\T^\prime_{\Sigma}$ consisting of all the directions where $u$ is not microlocally a $\Cl^\mathcal{M}$-vector with respect to $X$, it is denoted by $\WF_{\mathcal{M}}(u; X)$.
For every cone $\mathcal{C} \subset \R^m \setminus \{0\}$ and every subset $S \subset \Sigma \cap \Omega$ we define $\R\T^\prime_{S}(\Cl) = \{ (Z(x), {}^\mathrm{t} Z_x(x)^{-1} \xi) \; : \; Z(x) \in S, \, \xi \in \Cl \}$.
\end{defn}

\begin{thm} \label{thm:microlocal-FBI}
Let $u \in \mathcal{D}^\prime(\Sigma \cap \Omega)$. Are equivalent:
\begin{enumerate}
    \item The point $(0,\xi_0)$ does not belong to $\WF_{\mathcal{M}}(u; X)$;
    \item There exists an open convex cone $\Cl \subset \mathbb{R}^m \setminus \{0\}$ containing $\xi_0$, such that for every $\chi \in \Cl_c^\infty(\Sigma \cap \Omega)$, with $0 \leq \chi \leq 1$ and $\chi \equiv 1$ in some open neighborhood of the origin, there exist $V \Subset U$ a neighborhood of the origin and a constant $C > 0$ such that
    \begin{equation} \label{eq:FBI-decay}
        |\mathfrak{F}[\chi u](z,\zeta)| \leq 
        \frac{C^{k+1} M_k}{|\zeta|^k}, \qquad \forall (z,\zeta) \in \R\T^\prime_{Z(V)}(\Cl) \setminus 0.
    \end{equation}
\end{enumerate}

\end{thm}

\noindent Before proving Theorem \ref{thm:microlocal-FBI} we shall need the following lemma:

\comment{
\begin{lemma}\label{lem:inequality_derivatives_gaussian}
Let $\lambda>0$ and $\alpha\in\mathbb{Z}_+^m$. Then

\begin{equation}\label{eq:derivative-gaussian}
    \del_x^\alpha e^{-\lambda|x|^2}=\sum_{l=(l^\prime,l^{\prime\prime})}\frac{\alpha!}{l!}(-\lambda)^{|l|}(2x_1,\dots,2x_m)^{l^\prime}e^{-\lambda|x|^2},
\end{equation}

\noindent where the sum is taken on the set 

\begin{equation*}
    \{l=(l_1^1,l_2^1,\dots,l_1^m,l_2^m)\,:\,l_1^j+2l_2^j=j,\, j=1,\dots,m\},
\end{equation*}

\noindent and $l^\prime=(l_1^1,\dots,l_1^m)$ and $l^{\prime\prime}=(l_2^1,\dots,l_2^m)$.

\end{lemma}
}

\begin{lemma}
    Let $g \in \mathcal{C}^\infty_c(\Sigma \cap \Omega)$ and let $\Gamma, \mathcal{C} \subset \mathbb{R}^m \setminus 0$ be open and convex cones such that if $(v, \zeta) \in \Gamma \times \mathbb{R} \mathrm{T}^\prime_{\mathrm{supp}\, g}(\mathcal{C})$ then $v \cdot \Re \,  \zeta \leq 0$. Then for every $v \in \Gamma$,
    
    \begin{equation*}
        \lim_{\lambda \to 0}
        \int_{\mathbb{R} \mathrm{T}^\prime_{z}(\mathcal{C})} \mathfrak{F}^{\frac{1}{2}}[g](z + i\lambda v, \zeta) \mathrm{d} \zeta 
        = 
        \int_{\mathbb{R} \mathrm{T}^\prime_{z}(\mathcal{C})} \mathfrak{F}^{\frac{1}{2}}[g](z, \zeta) \mathrm{d} \zeta,
    \end{equation*}
    
    \noindent in the $\mathcal{C}^\infty(\Sigma \cap \Omega)$ topology.
\end{lemma}

\begin{proof}
Let $v \in \Gamma$, $\lambda > 0$, $z \in \Sigma \cap \Omega$, $\zeta \in \mathbb{R} \mathrm{T}^\prime_z(\mathcal{C})$, and $N \in\mathbb{Z}_+$. Then

\begin{align*}
    &(1 + \langle \zeta \rangle^2)^{N}  \mathfrak{F}^{\frac{1}{2}}[g](z + i\lambda v,\zeta) = \\
    &= \int_{U} \left \{ (1 - \Delta_X)^N e^{i \zeta \cdot (z + i\lambda v - Z(x))}  \right \} e^{-\frac{1}{2}\langle \zeta \rangle \langle z + i\lambda v - Z(x) \rangle^2} g(Z(x)) \Delta \left(\left(\frac{z + i\lambda v - Z(x)}{2} \right), \zeta \right) \mathrm{d} Z(x) \\
    & = \int_{U} e^{i \zeta \cdot (z + i\lambda v - Z(x))}  (1 - \Delta_X)^N \left \{e^{-\frac{1}{2}\langle \zeta \rangle \langle z + i\lambda v - Z(x) \rangle^2} g(Z(x)) \Delta \left(\left(\frac{z + i\lambda v - Z(x)}{2} \right), \zeta \right), \zeta)  \right \}  \mathrm{d} Z(x) \\
    & = \sum_{k=0}^N \binom{N}{k}(-1)^k  \int_{U} e^{i \zeta \cdot (z + i\lambda v - Z(x))}  \Delta_X^k \left \{ e^{-\frac{1}{2}\langle \zeta \rangle \langle z + i\lambda v - Z(x) \rangle^2} g(Z(x)) \Delta \left(\left(\frac{z + i\lambda v - Z(x)}{2} \right), \zeta \right) \right \} \mathrm{d} Z(x) \\
    & = \sum_{k=0}^N \binom{N}{k}(-1)^k \sum_{|\beta| = k} \frac{k!}{\beta!}\int_{U} e^{i \zeta \cdot (z + i\lambda v - Z(x))}  X^{2\beta} \left \{ e^{-\frac{1}{2}\langle \zeta \rangle \langle z + i\lambda v - Z(x) \rangle^2} g(Z(x)) \Delta \left(\left(\frac{z + i\lambda v - Z(x)}{2} \right), \zeta \right) \right \} \mathrm{d} Z(x) \\
    & = \sum_{k=0}^N \sum_{|\beta| = k} \sum_{\alpha \leq 2\beta}(-1)^k \binom{N}{k}  \frac{k!}{\beta!} \binom{2\beta}{\alpha} \int_{U} e^{i \zeta \cdot (z + i\lambda v - Z(x))}  X^{\alpha} e^{-\frac{1}{2}\langle \zeta \rangle \langle z + i\lambda v - Z(x) \rangle^2} X^{2\beta - \alpha} \bigg ( g(Z(x)) \Delta \left(\left(\frac{z + i\lambda v - Z(x)}{2} \right), \zeta \right) \bigg ) \mathrm{d} Z(x) \\
    & = \sum_{k=0}^N \sum_{|\beta| = k} \sum_{\alpha \leq 2\beta} \sum_{l^1_1 + 2 l^1_2 = \alpha_1} \cdots \sum_{l^m_1 + 2l^m_2 = \alpha_m}(-1)^k \binom{N}{k}  \frac{k!}{\beta!} \binom{2\beta}{\alpha}  \frac{\alpha!}{l^1_1! \cdots l^m_1!}  (-\langle \zeta \rangle)^{l^1_1 + l^1_2 \cdots + l^m_1 + l^m_2} \cdot \\
    & \quad \cdot \int_{U} e^{i \zeta \cdot (z + i\lambda v - Z(x))-\frac{1}{2}\langle \zeta \rangle \langle z + i\lambda v - Z(x) \rangle^2} (2(Z_1(x) - z_1 - i \lambda v_1))^{l^1_1} \cdots (2(Z_m(x) - z_m - i \lambda v_m))^{l^m_1} \cdot \\
    & \quad \cdot X^{2\beta - \alpha} \bigg ( g(Z(x)) \Delta \left(\left(\frac{z + i\lambda v - Z(x)}{2} \right), \zeta \right) \bigg ) \mathrm{d} Z(x).
\end{align*}

\noindent Now in view of 

\begin{equation*}
    \int_{-\infty}^\infty t^l e^{-at^2} \mathrm{d}t \leq \text{Const} \cdot \frac{1}{a^{\frac{l+1}{2}}},
\end{equation*}

\noindent we have that, if $\lambda |v| \leq \kappa^\prime c^\sharp / 4$,

\begin{align*}
    \left| (1 + \langle \zeta \rangle^2)^{N} \mathfrak{F}^\frac{1}{2}[g](z + i\lambda v,\zeta) \right| & \leq \sum_{k=0}^N \sum_{|\beta| = k} \sum_{\alpha \leq 2\beta} \sum_{l^1_1 + 2 l^1_2 = \alpha_1} \cdots \sum_{l^m_1 + 2l^m_2 = \alpha_m} \binom{N}{k}  \frac{k!}{\beta!} \binom{2\beta}{\alpha}  \frac{\alpha!}{l^1_1! \cdots l^m_1!} | \zeta|^{l^1_1 + l^1_2+ \cdots + l^m_1 + l^m_2} \cdot \\
    & \quad \cdot \int_{U} e^{-\frac{\kappa^\prime}{2} |z - Z(x)|^2 |\zeta|} |2(Z_1(x) - z_1 - i \lambda v_1)|^{l^1_1} \cdots |2(Z_m(x) - z_m - i \lambda v_m)|^{l^m_1}C^g_{2\beta - \alpha}| \mathrm{d} Z(x)| \\
    & \leq \sum_{k=0}^N \sum_{|\beta| = k} \sum_{\alpha \leq 2\beta} \sum_{l^1_1 + 2 l^1_2 = \alpha_1} \cdots \sum_{l^m_1 + 2l^m_2 = \alpha_m} \binom{N}{k}  \frac{k!}{\beta!} \binom{2\beta}{\alpha}  \frac{\alpha!}{l^1_1! \cdots l^m_1!}  | \zeta|^{l^1_1 + l^1_2 \cdots + l^m_1 + l^m_2} \frac{\text{Const}(g,\kappa^\prime, \alpha,\beta)}{|\zeta|^{\frac{l^1_1 + \cdots l^m_1 + m}{2}}} \\
    & = \sum_{k=0}^N \sum_{|\beta| = k} \sum_{\alpha \leq 2\beta} \sum_{l^1_1 + 2 l^1_2 = \alpha_1} \cdots \sum_{l^m_1 + 2l^m_2 = \alpha_m} \binom{N}{k}  \frac{k!}{\beta!} \binom{2\beta}{\alpha}  \frac{\alpha!}{l^1_1! \cdots l^m_1!}  | \zeta|^{\frac{l^1_1 + 2l^1_2 \cdots + l^m_1 + 2l^m_2}{2} - \frac{m}{2}} \text{Const}(g,\kappa^\prime, \alpha,\beta) \\
    & = \sum_{k=0}^N \sum_{|\beta| = k} \sum_{\alpha \leq 2\beta} \sum_{l^1_1 + 2 l^1_2 = \alpha_1} \cdots \sum_{l^m_1 + 2l^m_2 = \alpha_m} \binom{N}{k}  \frac{k!}{\beta!} \binom{2\beta}{\alpha}  \frac{\alpha!}{l^1_1! \cdots l^m_1!}  | \zeta|^{\frac{\alpha}{2} - \frac{m}{2}} \text{Const}(g,\kappa^\prime, \alpha,\beta) \\
    & \leq \text{Const}(g,\kappa^\prime, N) |\zeta|^{N - \frac{m}{2}}.
\end{align*}

\noindent Then we conclude that for every $N \in \mathbb{Z}_+$ there exist $C_N>0$ such that if $0 < \lambda \leq \frac{\kappa^\prime c^\sharp}{|v|}$ then 

\begin{equation*}
    \left| \mathfrak{F}^\frac{1}{2}[g](z + i \lambda v,\zeta) \right| \leq C_N (1+|\zeta|)^{-N}.
\end{equation*}

\noindent We can write the integral on the statement of the lemma as

\begin{equation*}
    \int_{\mathcal{C}} \mathfrak{F}^\frac{1}{2}[g](Z(x) + i\lambda v, {}^\mathrm{t} Z_x(x)^{-1}\xi) \det Z_x(x)^{-1} \mathrm{d} \xi.
\end{equation*}

\noindent So to prove the lemma it is enough to show that $X^\alpha \bigg( \mathfrak{F}^\frac{1}{2}[g](Z(x) + i\lambda v, {}^\mathrm{t} Z_x(x)^{-1}\xi) \det Z_x(x)^{-1} \bigg)$ is dominated by an integrable function in $\xi$, uniformly for small $\lambda > 0$, for every $\alpha \in \mathbb{Z}_+^m$. Arguing as before one can prove that for every $\alpha \in \mathbb{Z}_+^m$ and $N \in \mathbb{Z}_+$ there exists $C >0$ such that for every $0 \leq \lambda \leq \frac{\kappa^\prime c^\sharp}{|v|}$,

\begin{equation*}
    \left| X^\alpha \bigg( \mathfrak{F}^\frac{1}{2}[g](Z(x) + i\lambda v, {}^\mathrm{t} Z_x(x)^{-1}\xi) \det Z_x(x)^{-1} \bigg) \right| \leq C (1+|\zeta|)^{-N}.
\end{equation*}

\end{proof}

\begin{proof}[Proof of Theorem \ref{thm:microlocal-FBI}]
(1.) implies (2.)

Without loss of generality we shall assume that there exist $V\Subset U$ an open neighborhood of the origin, $\delta >0$, $\Gamma \subset \mathbb{R}^m \setminus 0$ an open convex cone, $\xi_0 \cdot \Gamma < 0$, and $f \in \mathcal{C}^1(\mathcal{W}_\delta(V,\Gamma))$ an $\mathcal{M}-$almost analytic function, with slow growth, such that 
\[
    u|_{Z(V)} = \mathrm{b}_\Gamma(f).
\]
So let $\chi \in \mathcal{C}^\infty_c(V)$, $0 \leq \chi \leq 1$, with $\chi \equiv 1$ in some open neighborhood of the origin, and let $v \in \Gamma \cap \mathbb{S}^{m-1}$ be fixed. Then the F.B.I. transform of $\chi u$ can written as

\begin{equation*}
    \mathfrak{F}[\chi u](z,\zeta) = \lim_{t \to 0^+} 
    \int_V e^{i\zeta\cdot(z-Z(x^\prime)) - \langle \zeta \rangle \langle z-Z(x^\prime)\rangle^2}
    \chi(x^\prime)f(Z(x^\prime)+itv)  \Delta(z-Z(x^\prime),\zeta)\, \d Z(x^\prime).
\end{equation*}
So let $a>0$ be such that $\xi_0 \cdot v = -a |\xi_0| $ and let $0 < r < a/(12\sqrt{2})$ be such that $\chi \equiv 1$ in $B_r(0) \subset V$.
Now fix $0<t<\delta/2$. We split the integral above in two:
\begin{align*}
    \int_V e^{i\zeta\cdot(z-Z(x^\prime)) - \langle \zeta \rangle \langle z-Z(x^\prime)\rangle^2} &
    \chi(x^\prime)f(Z(x^\prime)+itv)  \Delta(z-Z(x^\prime),\zeta)\, \d Z(x^\prime) =\\
    &= \int_{B_r(0)} e^{i\zeta\cdot(z-Z(x^\prime)) - \langle \zeta \rangle \langle z-Z(x^\prime)\rangle^2}
    f(Z(x^\prime)+itv)  \Delta(z-Z(x^\prime),\zeta)\, \d Z(x^\prime)\\
    &+\int_{V\setminus B_r(0)} e^{i\zeta\cdot(z-Z(x^\prime)) - \langle \zeta \rangle \langle z-Z(x^\prime)\rangle^2}
    \chi(x^\prime)f(Z(x^\prime)+itv)  \Delta(z-Z(x^\prime),\zeta)\, \d Z(x^\prime)
\end{align*}
Now we shall deal with these two integrals separately. The exponential in the second integral can be bounded by $e^{-\frac{r^2}{4}|\zeta|}$ if $z=Z(x)$ with $|x|<r/2$, so the absolute value of the second integral is bounded by a constant times $e^{-\frac{r^2}{4}|\zeta|}$, for all $(z,\zeta) \in \mathbb{R}\mathrm{T}^\prime_{Z(B_{r/2})}$. To estimate the first integral first we shall deform the contour of integration in the following manner:
\comment{
Let $(z,\zeta) \in \left.\R\T^\prime_{\Sigma}\right|_{\widetilde{V}}$ be fixed, where $\widetilde{V}\Subset V$ is a ball centered at the origin to be chosen latter. Let $\lambda > 0$ be such that the image of the map
}
\begin{equation*}
     Z(x^\prime) \mapsto Z(x^\prime)+i\frac{\delta}{2} v.
\end{equation*}
Stokes' theorem entails
\begin{align}
    \int_{B_r(0)} &e^{i\zeta\cdot(z-Z(x^\prime)) - \langle \zeta \rangle \langle z-Z(x^\prime)\rangle^2} 
    f(Z(x^\prime)+itv)  \Delta(z-Z(x^\prime),\zeta)\, \d Z(x^\prime) = \nonumber \\
     &\qquad \int_{B_r(0)}e^{i \zeta\cdot(z -Z(x^\prime) - i\delta/2v)- \langle \zeta \rangle \langle z -Z(x^\prime) - i\delta/2 v\rangle^2}f(Z(x^\prime) + i(t+\delta/2)v)\Delta(z -Z(x^\prime) -i\delta/2 v,\zeta)\mathrm{d}Z(x^\prime) \label{integral-I}\\ 
     &\;\;\, + \int_{\{ w \in \; Z(\partial B_r(0)) + i \sigma v, \sigma\in [0,\delta/2]\}} e^{i\zeta \cdot (z-w)-\langle \zeta \rangle \langle z-w\rangle^2}f(w + itv)\Delta(z - w,\zeta)\mathrm{d}w \label{integral-II}\\ 
     &\;\;\, - \int_{\{ w \in \; Z(B_{r}(0))  + i \sigma v, \sigma\in [0,\delta/2]\}} e^{i\zeta\cdot(z-w)-\langle \zeta \rangle \langle z-w\rangle^2}\Delta(z - w,\zeta)\sum_{k=1}^m\frac{\partial}{\partial \overline{w}_k}f(w + itv)\mathrm{d}\overline{w}_k\wedge\mathrm{d}w. \label{integral-III}
\end{align}
We start estimating the exponents. 
For $z,w \in Z(V)$, $\zeta \in \mathbb{R}\mathrm{T}^\prime_{\Sigma}|_z$ and $\sigma \in [0,\delta/2]$ we have
\begin{align*}
    \Re \big\{ i\zeta \cdot (z - w - i\sigma v) - \langle \zeta \rangle \langle z - w - i\sigma v \rangle^2 \big\} & = \Re \big\{ i \zeta \cdot (z - w) - \langle \zeta \rangle \langle z - w \rangle^2 \big\} + \sigma \Re \, \zeta \cdot v - \Re \big\{ \langle \zeta \rangle ( - \sigma^2|v|^2 - 2i\sigma (z-w)\cdot v ) \big\}\\
    & \leq - \kappa^\prime|\zeta||z-w|^2 + \sigma \Re \, \zeta \cdot v + |\zeta|(\sigma^2 + 2\sigma|z-w|).
\end{align*}

\noindent Now we shall define a cone $\mathcal{C}$. Recall that $\xi_0 \cdot v = - a|\xi_0|$, so there exists $0 < r^\prime \leq r/2$ and $\mathcal{C} \subset \mathbb{R}^m \setminus 0$ an open convex cone containing $\xi_0$ such that if $z \in Z(B_{r^\prime}(0))$ and $\zeta \in \mathbb{R}\mathrm{T}^\prime_{\Sigma}|_z(\mathcal{C})$ then
\[
    \Re \, \zeta \cdot v \leq -a/2|\zeta|.
\]
Thus, for $\zeta \in \mathbb{R}\mathrm{T}^\prime_\Sigma|_z(\mathcal{C})$ and $\delta < a/2$ we have
\begin{equation*}
    \Re \big\{ i\zeta \cdot (z - w - i\sigma v) - \langle \zeta \rangle \langle z - w - i\sigma v \rangle^2 \big\} \leq -\kappa^\prime|\zeta||z-w|^2 - |\zeta|\sigma ( a/4 -2|z - w|).
\end{equation*}

\noindent Since $C_\varphi < 1$, we have $|z - w|^2 \leq (9/2)r^2$, thus $2|z-w| < 3/\sqrt{2} r$.
The choice $r < a/(12\sqrt{2})$ implies $a/4 - 2|z-w| > a/8$, thus

\begin{equation*}
    \Re \big\{ i\zeta \cdot (z - w - i\sigma v) - \langle \zeta \rangle \langle z - w - i\sigma v \rangle^2 \big\} \leq -\kappa^\prime|\zeta||z-w|^2 - |\zeta|\frac{a\sigma}{8}.
\end{equation*}

\noindent In integral \eqref{integral-I} we have that $\sigma = \delta/2$, so we can bound it by a constant times $e^{-a\delta/16|\zeta|}$. For estimating \eqref{integral-II} we use that $|z - w|^2\geq r^2/4$, so we can bound $|\eqref{integral-II}|$ by a constant times $e^{-\frac{\kappa^\prime r^2}{4}|\zeta|}$. The exponent in integral \eqref{integral-III} is bounded by $e^{-\frac{a\sigma}{8}|\zeta|}$, so using the fact that $f$ is an $\mathcal{M}$-almost analytic function we can estimate $|\eqref{integral-III}|$ by a constant times

\begin{align*}
    \int_{0}^\infty e^{-\frac{a\sigma}{8}|\zeta|} C_f^{k+1}m_k (t+\sigma)^k|v|^k \mathrm{d}\sigma & \leq C_f^{k+1}|v|^k m_k \sum_{j=0}^k \binom{k}{j} t^{k-j} \int_0^\infty e^{-\frac{a|\zeta|}{8}\sigma}\sigma^j \mathrm{d}\sigma\\
    & = C_f^{k+1}|v|^k m_k \sum_{j=0}^k \binom{k}{j} t^{k-j} \left( \frac{8}{a|\zeta|} \right)^{j+1} j!\\
    & \leq C_f^{k+1}|v|^k \frac{8}{a|\zeta|} \left( t + \frac{8}{a|\zeta|} \right)^{k} m_k k!\\
    & \leq C_f^{k+1}|v|^k \frac{8}{a|\zeta|} \left( \delta/2 + \frac{8}{a|\zeta|} \right)^{k} M_k,
\end{align*}

\noindent for every $k \in \mathbb{Z}_+$. Summing all these estimates one obtain \eqref{eq:FBI-decay} with $V = B_{r/2}(0)$.

\vspace{0.2cm}

(2.) implies (1.)\\

Let $\mathcal{C}\subset \mathbb{R}^m\setminus 0$, $\chi \in \mathcal{C}^\infty_c(\Sigma\cap\Omega)$ and $V\Subset U$ be as in (2.). In order to use an exact inversion formula for the F.B.I. transform we shall extend the function $\varphi$ to the whole $\mathbb{R}^m$ in an appropriate manner. We shall replace $\varphi(x)$ by $\Psi(x)\varphi(x)$, where $\Psi \in \mathcal{C}_c^\infty(B_{r}(0))$, $\Psi \equiv 1$ in $B_{r/2}(0)$, $B_r(0) \Subset U$, and $\| \Psi^\prime \| \leq \text{Const} \cdot r^{-1}$, for some small $r$ so the image of $\mathbb{R}^m \ni x \mapsto x + i \Psi(x)\varphi(x)$ is (globally) well-positioned. From now on we shall write $\varphi$ instead of $\Psi \varphi$. Without loss of generality we can assume that the support of $\chi$ and $V$ are contained in $Z(B_{r/2}(0))$, and we shall assume that $V=B_{r^\prime}(0)$. We can then use the following inversion formula (see Lemma IX.4.1 of \cite{trevesbook} and Theorem 3.3 of \cite{braun:2022}):

\begin{equation} \label{eq:inversion_formula_FBI}
 \chi(x)u(x) = \lim_{\epsilon \to 0^+} \frac{1}{(2 \pi^3)^{\frac{m}{2}}}\iint_{\mathbb{R}\mathrm{T}^\prime_{Z(\mathbb{R}^m)}} e^{i \zeta \cdot (Z(x) - z^\prime) -\langle \zeta \rangle \langle Z(x) - z^\prime \rangle^2 - \epsilon \langle \zeta \rangle^2} \mathfrak{F}[\chi u](z^\prime,\zeta)\langle \zeta \rangle^{\frac{m}{2}}\mathrm{d}z^\prime\mathrm{d}\zeta. 
\end{equation}

\noindent For every $\epsilon>0$ we set

\begin{equation*}
    v_1^\epsilon(x) = \frac{1}{(2 \pi^3)^{\frac{m}{2}}}\iint_{\mathbb{R}\mathrm{T}^\prime_{Z(\mathbb{R}^m)}(\mathcal{C})} e^{i \zeta \cdot (Z(x) - z^\prime) -\langle \zeta \rangle \langle Z(x) - z^\prime \rangle^2 - \epsilon \langle \zeta \rangle^2} \mathfrak{F}[\chi u](z^\prime,\zeta)\langle \zeta \rangle^{\frac{m}{2}}\mathrm{d}z^\prime\mathrm{d}\zeta,
\end{equation*}

\noindent and

\begin{equation*}
    v_2^\epsilon(z) = \frac{1}{(2 \pi^3)^{\frac{m}{2}}}\iint_{\mathbb{R}\mathrm{T}^\prime_{Z(\mathbb{R}^m)}(\mathbb{R}^m \setminus \mathcal{C})} e^{i \zeta \cdot (z - z^\prime) -\langle \zeta \rangle \langle z - z^\prime \rangle^2 - \epsilon \langle \zeta \rangle^2} \mathfrak{F}[\chi u](z^\prime,\zeta)\langle \zeta \rangle^{\frac{m}{2}}\mathrm{d}z^\prime\mathrm{d}\zeta.
\end{equation*}

\noindent We shall prove that $v_1^\epsilon(x)$ converges to an $\mathcal{C}^{\mathcal{M}}(V_0; \mathrm{X})$, for some $V_0 \subset V$ open neighborhood of the origin, and that $v_2^\epsilon(z)$ converges to a sum of holomorphic function defined on wedges. We start with $v_1^\epsilon$. Let $|x| < r^\prime/2$. If $(z^\prime,\zeta) \in \mathbb{R}\mathrm{T}^\prime_{Z(\mathbb{R}^m \setminus V)}(\mathcal{C})$ then 

\begin{equation*}
    \left| e^{i \zeta \cdot (z - z^\prime) -\langle \zeta \rangle \langle z - z^\prime \rangle^2 - \epsilon \langle \zeta \rangle^2} \mathfrak{F}[\chi u](z^\prime,\zeta)\langle \zeta \rangle^{\frac{m}{2}} \right| \leq \text{Const}\cdot e^{-\epsilon_1|\zeta| - \frac{\kappa^\prime|\zeta|}{2}|z-z^\prime|^2},
\end{equation*}

\noindent for some $\epsilon_1>0$. Now if $(z^\prime,\zeta) \in \mathbb{R}\mathrm{T}^\prime_{Z(V)}(\mathcal{C})$ then

\begin{equation*}
    \left| e^{i \zeta \cdot (z - z^\prime) -\langle \zeta \rangle \langle z - z^\prime \rangle^2 - \epsilon \langle \zeta \rangle^2} \mathfrak{F}[\chi u](z^\prime,\zeta)\langle \zeta \rangle^{\frac{m}{2}} \right| \leq \text{Const}^{k+1}\frac{M_k}{|\zeta|^k} e^{- \kappa^\prime|\zeta||z-z^\prime|^2} ,
\end{equation*}

\noindent for all $k \in \mathbb{Z}_+$. Combining these two estimates we get that

\begin{equation*}
    \left| e^{i \zeta \cdot (z - z^\prime) -\langle \zeta \rangle \langle z - z^\prime \rangle^2 - \epsilon \langle \zeta \rangle^2} \mathfrak{F}[\chi u](z^\prime,\zeta)\langle \zeta \rangle^{\frac{m}{2}} \right| \leq C_1^{k+1}\frac{M_k}{|\zeta|^k}e^{- \frac{\kappa^\prime|\zeta|}{2}|z-z^\prime|^2},
\end{equation*}
for all $(z^\prime,\zeta) \in \mathbb{R}\mathrm{T}^\prime_{Z(\mathbb{R}^m)}(\mathcal{C})$ and $k \in \mathbb{Z}_+$, where $C_1>0$. Thus we have that

\begin{align*}
    \mathrm{X}^\alpha v_1^\epsilon(x)&= \sum_{\beta\leq\alpha}\binom{\alpha}{\beta}\iint_{\mathbb{R}\mathrm{T}^\prime_{Z(\mathbb{R}^m)(\mathcal{C})}}\mathrm{X}^{\alpha-\beta} e^{i\zeta\cdot(Z(x)-z^\prime)}\mathrm{X}^\beta e^{-\langle\zeta\rangle\langle Z(x)-z^\prime\rangle^2}\cdot\\
    &\cdot e^{-\epsilon\langle\zeta\rangle^2}\mathfrak{F}[\chi u](z^\prime,\zeta)\langle\zeta\rangle^\frac{m}{2}\d\zeta\d Z^\prime\\
    &=\sum_{\beta\leq\alpha}\binom{\alpha}{\beta}\sum_{l^1_1+2l^1_2=\beta_1}\cdots\sum_{l^m_1+2\l^m_2=\beta_m}\frac{\beta!}{l^1_1!l^1_2!\cdots l^m_1!l^m_2!}\cdot\\
    &\cdot \iint_{\mathbb{R}\mathrm{T}^\prime_{Z(\mathbb{R}^m)(\mathcal{C})}} e^{i\zeta\cdot(Z(x)-z^\prime)-\langle\zeta\rangle\langle Z(x)-z^\prime\rangle^2-\epsilon\langle\zeta\rangle^2}\mathfrak{F}[\chi u](z^\prime,\zeta)\langle\zeta\rangle^\frac{m}{2}\cdot\\
    &\cdot(-\langle\zeta\rangle)^{l^1_1+l^1_2+\cdots+ l^m_1+l^m_2}(i\zeta)^{\alpha-\beta}(2(Z_1(x,t)-z^\prime_1))^{l^1_1}\cdots\\
    &\cdots(2(Z_m(x)-z^\prime_m))^{l^m_1}\d\zeta\d Z^\prime.\\
\end{align*}

\noindent Therefore if $|x| < r^\prime/2$,

\begin{align*}
   \left| \mathrm{X}^\alpha v_1^\epsilon(x)\right|&\leq \sum_{\beta\leq\alpha}\binom{\alpha}{\beta}\sum_{l=(l^\prime,l^{\prime\prime})}\frac{\beta!}{l!}\iint_{\mathbb{R}\mathrm{T}^\prime_{Z(\mathbb{R}^m)(\mathcal{C})}}\left| e^{i \zeta \cdot (z - z^\prime) -\langle \zeta \rangle \langle z - z^\prime \rangle^2 - \epsilon \langle \zeta \rangle^2} \mathfrak{F}[\chi u](z^\prime,\zeta)\langle \zeta \rangle^{\frac{m}{2}} \right| \cdot \\
   &\cdot |\zeta|^{|\alpha-\beta|+l^1_1+l^1_2+\cdots+ l^m_1+l^m_2}|\d\zeta\d Z^\prime|\\
   &\leq \text{Const}^{|\alpha|+1}\sum_{\beta\leq\alpha}\binom{\alpha}{\beta}\sum_{l=(l^\prime,l^{\prime\prime})}\frac{\beta!}{l!}\iint_{\mathbb{R}\mathrm{T}^\prime_{Z(\mathbb{R}^m)(\mathcal{C})}}C_1^{k+1}e^{- \frac{\kappa^\prime|\zeta|}{2}|z-z^\prime|^2}\frac{M_k}{|\zeta|^k}|\zeta|^{|\alpha-\beta|+l^1_1+l^1_2+\cdots+ l^m_1+l^m_2}|\d\zeta\d Z^\prime|\\
   \end{align*}
   
   \noindent so choosing $k=|\alpha-\beta|+l^1_1+l^1_2+\cdots+ l^m_1+l^m_2+\kappa$, where $\kappa$ is any integer bigger than $3/2m+1$, we obtain
   
   \begin{align*}
       \left| \mathrm{X}^\alpha v_1^\epsilon(x)\right|&\leq  C_2^{|\alpha|+1}\sum_{\beta\leq\alpha}\binom{\alpha}{\beta}\sum_{l^1_1+2l^1_2=\beta_1}\cdots\sum_{l^m_1+2l^m_2=\beta_m}\frac{\beta!}{l^1_1!l^1_2!\cdots l^m_1!l^m_2!}M_{\kappa}M_{|\alpha|-|\beta|}M_{l^1_1+l^1_2}\cdots M_{l^m_1+l^m_2}\\
       &\leq C_3^{|\alpha|+1}\sum_{\beta\leq\alpha}\binom{\alpha}{\beta}\sum_{l^1_1+2l^1_2=\beta_1}\cdots\sum_{l^m_1+2l^m_2=\beta_m}\frac{\beta!}{l^1_1!(2l^1_2)!\cdots l^m_1!(2l^m_2)!}M_{|\alpha|-|\beta|}M_{l^1_1+l^1_2}\frac{(2l^1_2)!}{l^1_2!}\cdots M_{l^m_1+l^m_2}\frac{(2l^m_2)!}{l^m_2!}\\
       &\leq  C_4^{|\alpha|+1}\sum_{\beta\leq\alpha}\binom{\alpha}{\beta}\sum_{l^1_1+2l^1_2=\beta_1}\cdots\sum_{l^m_1+2l^m_2=\beta_m}\frac{\beta!}{l^1_1!(2l^1_2)!\cdots l^m_1!(2l^m_2)!}M_{|\alpha|-|\beta|}M_{l^1_1+l^1_2}l^1_2!\cdots M_{l^m_1+l^m_2} l^m_2!\\
   &\leq C_4^{|\alpha|+1}\sum_{\beta\leq\alpha}\binom{\alpha}{\beta}\sum_{l^1_1+2l^1_2=\beta_1}\cdots\sum_{l^m_1+2l^m_2=\beta_m}\frac{\beta!}{l^1_1!(2l^1_2)!\cdots l^m_1!(2l^m_2)!}M_{|\alpha|-|\beta|}M_{l^1_1+l^1_2}M_{l^1_2}\cdots M_{l^m_1+l^m_2} M_{l^m_2}\\
   &\leq  C_5^{|\alpha|+1}\sum_{\beta\leq\alpha}\binom{\alpha}{\beta}\sum_{l^1_1+2l^1_2=\beta_1}\cdots\sum_{l^m_1+2l^m_2=\beta_m}\frac{\beta!}{l^1_1!(2l^1_2)!\cdots l^m_1!(2l^m_2)!}M_{|\alpha|-|\beta|}M_{l^1_1+2l^1_2}\cdots M_{l^m_1+2l^m_2}\\
   &\leq  C_6^{|\alpha|+1}\sum_{\beta\leq\alpha}\binom{\alpha}{\beta}\sum_{l^1_1+2l^1_2=\beta_1}\cdots\sum_{l^m_1+2l^m_2=\beta_m}\frac{\beta!}{l^1_1!(2l^1_2)!\cdots l^m_1!(2l^m_2)!}M_{|\alpha|-|\beta|}M_{|\beta|}\\
   &\leq  C_6^{|\alpha|+1}M_{|\alpha|}\sum_{\beta\leq\alpha}\binom{\alpha}{\beta}\sum_{l^1_1+2l^1_2=\beta_1}\cdots\sum_{l^m_1+2l^m_2=\beta_m}\frac{\beta!}{l^1_1!(2l^1_2)!\cdots l^m_1!(2l^m_2)!}\\
   &\leq  C_6^{|\alpha|+1}M_{|\alpha|}\sum_{\beta\leq\alpha}\binom{\alpha}{\beta}\sum_{l^1_1+l^1_2=\beta_1}\cdots\sum_{l^m_1+l^m_2=\beta_m}\frac{\beta!}{l^1_1!l^1_2!\cdots l^m_1!l^m_2!}\\
   &\leq  C_7^{|\alpha|+1}M_{|\alpha|},\\
   \end{align*}
   
   \noindent where the constant $C_4$ does not depend on $\epsilon$ (see Lemma $4.2$ of \cite{bierstone} for estimating this binomials). 
   Now we shall deal with $v_2^\epsilon(z)$. 
   As usual we start writing $\mathbb{R}^m \setminus \mathcal{C} = \bigcup_{j=1}^N \overline{\mathcal{C}_j}$, where the $\mathcal{C}_j \subset \mathbb{R}^m \setminus 0$ are open convex cones, pair-wise disjoints, and such that the sets
   
   \begin{equation*}
       \Gamma_j \doteq \{ v \in \mathbb{R}^m\setminus 0 \; : \; v \cdot \xi_0 < 0 \text{ and } v \cdot \Cl_j > 0 \},
   \end{equation*}
   
   \noindent are open, convex, cones. Also we can assume, shrinking if necessary the cones $\Gamma_j$, 
   that there exist $c^\sharp > 0$ and $0 < \tilde{r} < r^\prime$ such that for every $j = 1, \dots, N$, and $(v,\zeta) \in \Gamma_j \times \mathbb{R} \mathrm{T}^\prime_{Z(B_{\tilde{r}}(0))}(\mathcal{C}_j)$ then
    
\comment{
    \color{red}
   \begin{equation*}
       v \cdot \xi > 2c^\sharp |v||\xi|,\quad \forall (v,\xi) \in \Gamma_j \times \mathcal{C}_j   \end{equation*}
\color{black}
}

   \begin{equation*}
       v \cdot \Re \, \zeta \geq c^\sharp |v||\zeta|.
   \end{equation*}
   
   \noindent So for every $j = 1, \dots, N$, we set
   
   \begin{equation*}
       v_{2j}^\epsilon(z) = \frac{1}{(2 \pi^3)^{\frac{m}{2}}}\iint_{\mathbb{R}\mathrm{T}^\prime_{Z(\mathbb{R}^m)}(\mathcal{C}_j)} e^{i \zeta \cdot (z - z^\prime) -\langle \zeta \rangle \langle z - z^\prime \rangle^2 - \epsilon \langle \zeta \rangle^2} \mathfrak{F}[\chi u](z^\prime,\zeta)\langle \zeta \rangle^{\frac{m}{2}}\mathrm{d}z^\prime\mathrm{d}\zeta.
   \end{equation*}
   
   \noindent We claim that there exist an $\delta>0$ such that $v_{2j}^\epsilon(z)$ is holomorphic and uniformly bounded on the compact sets of $\mathcal{W}_\delta(V,\Gamma_j)$. Now writing $z = Z(x) + i v$, with $|v| \leq \delta0,$ and $z^\prime = Z(x^\prime)$ we have that if $|x^\prime| < \tilde{r}$,
   
   \begin{align*}
       \left| e^{i \zeta \cdot (z - z^\prime) - \langle \zeta \rangle \langle z - z^\prime \rangle^2} \right| & = \left| e^{i \zeta \cdot (Z(x) - z^\prime) - \langle \zeta \rangle \langle Z(x) - z^\prime \rangle^2} e^{-  \zeta \cdot v - \langle \zeta \rangle ( 2i(Z(x) - z^\prime) \cdot v - |v|^2)} \right|\\
       & \leq e^{-\kappa^\prime |\zeta||Z(x) - z^\prime|^2 - v \cdot \Re \zeta + |\zeta|(2|Z(x)-z^\prime||v| + |v|^2)}\\
       & \leq e^{-\kappa^\prime |\zeta||Z(x) - z^\prime|^2 - c^\sharp |v||\zeta| + |\zeta|(2|Z(x)-z^\prime||v| + |v|^2)} \\
       & = e^{-|\zeta||Z(x) - z^\prime|(\kappa^\prime |Z(x) - z^\prime| - 2|v|) - |v||\zeta| (c^\sharp - |v|) }\\
       &  \leq e^{|\zeta|\frac{|v|^2}{\kappa^\prime} - |\zeta| \frac{c^\sharp |v|}{2}}\\
       &  = e^{-|\zeta||v| ( \frac{c^\sharp}{2} - \frac{|v|}{\kappa^\prime}) }\\
       & \leq e^{-\frac{c^\sharp}{4}|v||\zeta|},
  \end{align*}
  
  \noindent assuming that $\delta < \frac{\kappa^\prime c^\sharp}{4}$. Now if $|x^\prime| \geq \tilde{r}$ and $|x| < \tilde{r}/2$, 
  
  \begin{align*}
      \left|  e^{i \zeta \cdot (z - z^\prime) - \langle \zeta \rangle \langle z - z^\prime \rangle^2} \right| & \leq e^{-\kappa^\prime |\zeta||Z(x) - z^\prime|^2 - v \cdot \Re \zeta + |\zeta|(2|Z(x)-z^\prime||v| + |v|^2)}\\
      & \leq e^{-\kappa^\prime \frac{\tilde{r}^2}{8}|\zeta| -\frac{\kappa^\prime}{2}|Z(x) - z^\prime|^2|\zeta| + |\zeta|(|v|+|v|^2 + 2|Z(x) - z^\prime||v|)}\\
      & =  e^{-\kappa^\prime \frac{\tilde{r}^2}{16}|\zeta| -|\zeta||Z(x) - z^\prime|\left(\frac{\kappa^\prime}{2}|Z(x) - z^\prime| - 2|v|\right) - |\zeta|(\kappa^\prime \frac{\tilde{r}^2}{16} - |v| - |v|^2)} \\ 
      & \leq e^{-\kappa^\prime \frac{\tilde{r}^2}{16}|\zeta|},
  \end{align*}
  
  \noindent if we assume $\delta \leq \min \left \{ \frac{\kappa^\prime \tilde{r}}{8}, \frac{\sqrt{1+\frac{\kappa^\prime \tilde{r}^2}{4}}}{2} \right \}$. Therefore if $\delta < \min \left \{ \frac{\kappa^\prime c^\sharp}{4}, \frac{\kappa^\prime \tilde{r}}{8}, \frac{\sqrt{1+\frac{\kappa^\prime \tilde{r}^2}{4}}}{2} \right \}$, $v_{2j}^\epsilon(z)$ is uniformly bounded on the compact sets of $\mathcal{W}_\delta(B_{\tilde{r}/2}, \Gamma_j)$, and thus there exists a sequence $\epsilon_k \to 0$ such that $v_{2j}^{\epsilon_k}(z) \rightarrow v_{2j}(z)$, which is holomorphic on $\mathcal{W}_\delta(B_{\tilde{r}/2}(0), \Gamma_j)$. Now we claim that $\lim_{\epsilon \to 0^+} v_{2j}^\epsilon (Z(x)) = \mathrm{bv}_{\Gamma_j}(v_{2j}(z))$ on $Z(B_{\tilde{r}/2}(0))$. So let $g \in \mathcal{C}^\infty_c(Z(B_{\tilde{r}/2}(0)))$. Then
  
  \begin{align*}
      &\langle v_{2j}^\epsilon(Z(x)), g(Z(x)) \rangle_{\mathcal{D}^\prime(\Sigma)}  = \\ 
      & = \int_{B_{\tilde{r}/2}(0)} v_{2j}^\epsilon(Z(x)) g(Z(x)) \mathrm{d} Z(x) \\
      & = \frac{1}{(2 \pi^3)^{\frac{m}{2}}}\int_{B_{\tilde{r}/2}(0)} \iint_{\mathbb{R}\mathrm{T}^\prime_{Z(\mathbb{R}^m)}(\mathcal{C}_j)} e^{i \zeta \cdot (Z(x) - z^\prime) -\langle \zeta \rangle \langle Z(x) - z^\prime \rangle^2 - \epsilon \langle \zeta \rangle^2} \mathfrak{F}[\chi u](z^\prime,\zeta)\langle \zeta \rangle^{\frac{m}{2}} g(Z(x)) \mathrm{d}z^\prime\mathrm{d}\zeta \mathrm{d} Z(x) \\
      & = \frac{1}{(2 \pi^3)^{\frac{m}{2}}}\int_{B_{\tilde{r}/2}(0)} \int_{\mathbb{R}^m} \int_{\mathbb{R}\mathrm{T}^\prime_{Z(x^\prime)}(\mathcal{C}_j)} e^{i \zeta \cdot (Z(x) - Z(x^\prime)) -\langle \zeta \rangle \langle Z(x) - Z(x^\prime) \rangle^2 - \epsilon \langle \zeta \rangle^2} \mathfrak{F}[\chi u](Z(x^\prime),\zeta)\langle \zeta \rangle^{\frac{m}{2}} g(Z(x)) \mathrm{d}\zeta  \mathrm{d} Z(x^\prime) \mathrm{d} Z(x) \\ 
      & =  \frac{1}{(2 \pi^3)^{\frac{m}{2}}}\int_{B_{\tilde{r}/2}(0)} \int_{\mathbb{R}^m} \int_{\mathbb{R}\mathrm{T}^\prime_{Z(x^\prime)}(\mathcal{C}_j)} \Bigg \langle \chi(y)u(Z(y)), e^{i \zeta \cdot (Z(x^\prime) - Z(y)) - \langle \zeta \rangle \langle Z(x^\prime) - Z(y) \rangle^2} \Delta(Z(x^\prime) - Z(y),\zeta) \cdot \\ 
      & \qquad \cdot e^{i \zeta \cdot (Z(x) - Z(x^\prime)) -\langle \zeta \rangle \langle Z(x) - Z(x^\prime) \rangle^2 - \epsilon \langle \zeta \rangle^2} \langle \zeta \rangle^{\frac{m}{2}}  \Bigg \rangle_{\mathcal{D}^\prime(\Sigma)} g(Z(x)) \mathrm{d}\zeta \mathrm{d} Z(x^\prime) \mathrm{d} Z(x) \\ 
      & = \Bigg \langle \chi(y)u(Z(y)), \frac{1}{(2 \pi^3)^{\frac{m}{2}}}\int_{B_{\tilde{r}/2}(0)} \int_{\mathbb{R}^m} \int_{\mathbb{R}\mathrm{T}^\prime_{Z(x^\prime)}(\mathcal{C}_j)} e^{i \zeta \cdot (Z(x) - Z(y)) - \langle \zeta \rangle \left( \langle Z(x^\prime) - Z(y) \rangle^2 + \langle Z(x) - Z(x^\prime) \rangle^2 \right) - \epsilon \langle \zeta \rangle^2} \cdot \\ 
      & \qquad \cdot \Delta(Z(x^\prime) - Z(y), \zeta) \langle \zeta \rangle^{\frac{m}{2}} g(Z(x)) \mathrm{d}\zeta \mathrm{d} Z(x^\prime) \mathrm{d} Z(x) \Bigg \rangle_{\mathcal{D}^\prime(\Sigma)} \\ 
      & = \Bigg \langle \chi(y)u(Z(y)), \frac{1}{(2 \pi^3)^{\frac{m}{2}}}\int_{B_{\tilde{r}/2}(0)} \int_{\mathbb{R}^m} \int_{\mathbb{R}\mathrm{T}^\prime_{Z(y)}(\mathcal{C}_j)} e^{i \zeta \cdot (Z(x) - Z(y)) - \langle \zeta \rangle \left( \langle Z(x^\prime) - Z(y) \rangle^2 + \langle Z(x) - Z(x^\prime) \rangle^2 \right) - \epsilon \langle \zeta \rangle^2} \cdot \\ 
      & \qquad \cdot \Delta(Z(x^\prime) - Z(y), \zeta) \langle \zeta \rangle^{\frac{m}{2}} g(Z(x)) \mathrm{d}\zeta \mathrm{d} Z(x^\prime) \mathrm{d} Z(x) \Bigg \rangle_{\mathcal{D}^\prime(\Sigma)} \\
      & = \Bigg \langle \chi(y)u(Z(y)), \frac{1}{(2 \pi^3)^{\frac{m}{2}}}\int_{B_{\tilde{r}/2}(0)}  \int_{\mathbb{R}\mathrm{T}^\prime_{Z(y)}(\mathcal{C}_j)} \Bigg \{ \int_{\mathbb{R}^m} e^{- \langle \zeta \rangle \left( \langle Z(x^\prime) - Z(y) \rangle^2 + \langle Z(x) - Z(x^\prime) \rangle^2 \right)} \langle \zeta \rangle^{\frac{m}{2}} \Delta(Z(x^\prime) - Z(y), \zeta) \mathrm{d} Z(x^\prime) \Bigg \} \cdot \\
      & \qquad \cdot e^{i \zeta \cdot (Z(x) - Z(y)) -\epsilon \langle \zeta \rangle^2}  g(Z(x)) \mathrm{d}\zeta \mathrm{d} Z(x) \Bigg \rangle_{\mathcal{D}^\prime(\Sigma)}.
  \end{align*}
  
  \noindent We shall use the following identity:
  
  \begin{equation*}
      \frac{\langle \zeta \rangle ^{\frac{m}{2}}}{\pi^\frac{m}{2}} \int_{Z(\mathbb{R}^m)} e^{\langle \zeta \rangle \left( \langle z^\prime - \omega \rangle^2 + \langle z - z^\prime \rangle^2 \right)} \Delta(z^\prime - \omega, \zeta) \mathrm{d}z^\prime = \frac{e^{-\frac{1}{2}}\langle \zeta \rangle \langle z - \omega \rangle^2}{2^\frac{m}{2}} \Delta \left(\frac{z-\omega}{2}, \zeta \right),
  \end{equation*}
  
  \noindent for every $z,\omega \in \mathbb{C}^m$. Thus 
  
  \begin{align*}
    &\langle v_{2j}^\epsilon(Z(x)), g(Z(x)) \rangle_{\mathcal{D}^\prime(\Sigma)}  = \\ & = \Bigg \langle \chi(y)u(Z(y)), \frac{1}{(2 \pi)^{m}}\int_{B_{\tilde{r}/2}(0)}  \int_{\mathbb{R}\mathrm{T}^\prime_{Z(y)}(\mathcal{C}_j)} e^{i \zeta \cdot (Z(x) - Z(y)) - \frac{1}{2} \langle \zeta \rangle \langle Z(x) - Z(y) \rangle^2 -\epsilon \langle \zeta \rangle^2} \Delta \left(\frac{Z(x)-Z(y)}{2}, \zeta \right) g(Z(x)) \mathrm{d}\zeta \mathrm{d} Z(x) \Bigg \rangle_{\mathcal{D}^\prime(\Sigma)}\\
    & = \Bigg \langle \chi(y)u(Z(y)), \frac{1}{(2 \pi)^{m}} \int_{\mathbb{R}\mathrm{T}^\prime_{Z(y)}(\mathcal{C}_j)} e^{-\epsilon \langle \zeta \rangle^2} \mathfrak{F}^\frac{1}{2}[g](Z(y), -\zeta) \mathrm{d}\zeta \Bigg \rangle_{\mathcal{D}^\prime(\Sigma)} \\
    & \underset{\epsilon \to 0^+}{\longrightarrow} \Bigg \langle \chi(y)u(Z(y)), \frac{1}{(2 \pi)^{m}} \int_{\mathbb{R}\mathrm{T}^\prime_{Z(y)}(\mathcal{C}_j)} \mathfrak{F}^\frac{1}{2}[g](Z(y), -\zeta) \mathrm{d}\zeta \Bigg \rangle_{\mathcal{D}^\prime(\Sigma)}.
  \end{align*}
  
  \noindent On the other hand, we have that
  
  \begin{align*}
      &\langle \mathrm{bv}_{\Gamma_j}(v_{2j}(z)), g \rangle_{\mathcal{D}^\prime(\Sigma)} =\\
      &= \lim_{\lambda \to 0^+} \int_{B_{\tilde{r}/2}(0)} v_{2j} (Z(x) + i\lambda v) g(Z(x)) \mathrm{d} Z(x) \\ 
      & = \lim_{\lambda \to 0^+} \frac{1}{(2 \pi^3)^{\frac{m}{2}}}\int_{B_{\tilde{r}/2}(0)} \iint_{\mathbb{R}\mathrm{T}^\prime_{Z(\mathbb{R}^m)}(\mathcal{C}_j)} e^{i \zeta \cdot (Z(x) + i \lambda v - z^\prime) -\langle \zeta \rangle \langle Z(x) + i\lambda v - z^\prime \rangle^2} \mathfrak{F}[\chi u](z^\prime,\zeta)\langle \zeta \rangle^{\frac{m}{2}} g(Z(x)) \mathrm{d}z^\prime\mathrm{d}\zeta \mathrm{d} Z(x) \\
      & =  \lim_{\lambda \to 0^+} \frac{1}{(2 \pi^3)^{\frac{m}{2}}}\int_{B_{\tilde{r}/2}(0)} \int_{\mathbb{R}^m} \int_{\mathbb{R}\mathrm{T}^\prime_{Z(x^\prime)}(\mathcal{C}_j)} e^{i \zeta \cdot (Z(x) + i \lambda v - Z(x^\prime)) -\langle \zeta \rangle \langle Z(x) + i \lambda v - Z(x^\prime) \rangle^2} \mathfrak{F}[\chi u](Z(x^\prime),\zeta)\langle \zeta \rangle^{\frac{m}{2}} g(Z(x)) \mathrm{d}\zeta  \mathrm{d} Z(x^\prime) \mathrm{d} Z(x) \\ 
      & =  \lim_{\lambda \to 0^+}\frac{1}{(2 \pi^3)^{\frac{m}{2}}}\int_{B_{\tilde{r}/2}(0)} \int_{\mathbb{R}^m} \int_{\mathbb{R}\mathrm{T}^\prime_{Z(x^\prime)}(\mathcal{C}_j)} \Bigg \langle \chi(y)u(Z(y)), e^{i \zeta \cdot (Z(x^\prime) - Z(y)) - \langle \zeta \rangle \langle Z(x^\prime) - Z(y) \rangle^2} \Delta(Z(x^\prime) - Z(y),\zeta) \cdot \\ 
      & \qquad \cdot e^{i \zeta \cdot (Z(x) + i \lambda v - Z(x^\prime)) -\langle \zeta \rangle \langle Z(x) + i \lambda v - Z(x^\prime) \rangle^2} \langle \zeta \rangle^{\frac{m}{2}}  \Bigg \rangle_{\mathcal{D}^\prime(\Sigma)} g(Z(x)) \mathrm{d}\zeta \mathrm{d} Z(x^\prime) \mathrm{d} Z(x) \\ 
      & =  \lim_{\lambda \to 0^+} \Bigg \langle \chi(y)u(Z(y)), \frac{1}{(2 \pi^3)^{\frac{m}{2}}}\int_{B_{\tilde{r}/2}(0)} \int_{\mathbb{R}^m} \int_{\mathbb{R}\mathrm{T}^\prime_{Z(x^\prime)}(\mathcal{C}_j)} e^{i \zeta \cdot (Z(x) + i \lambda v - Z(y)) - \langle \zeta \rangle \left( \langle Z(x^\prime) - Z(y) \rangle^2 + \langle Z(x) + i \lambda v - Z(x^\prime) \rangle^2 \right)} \cdot \\ 
      & \qquad \cdot \Delta(Z(x^\prime) - Z(y), \zeta) \langle \zeta \rangle^{\frac{m}{2}} g(Z(x)) \mathrm{d}\zeta \mathrm{d} Z(x^\prime) \mathrm{d} Z(x) \Bigg \rangle_{\mathcal{D}^\prime(\Sigma)} \\ 
      & =  \lim_{\lambda \to 0^+}\Bigg \langle \chi(y)u(Z(y)), \frac{1}{(2 \pi^3)^{\frac{m}{2}}}\int_{B_{\tilde{r}/2}(0)} \int_{\mathbb{R}^m} \int_{\mathbb{R}\mathrm{T}^\prime_{Z(y)}(\mathcal{C}_j)} e^{i \zeta \cdot (Z(x) + i \lambda v - Z(y)) - \langle \zeta \rangle \left( \langle Z(x^\prime) - Z(y) \rangle^2 + \langle Z(x) + i \lambda v - Z(x^\prime) \rangle^2 \right)} \cdot \\ 
      & \qquad \cdot \Delta(Z(x^\prime) - Z(y), \zeta) \langle \zeta \rangle^{\frac{m}{2}} g(Z(x)) \mathrm{d}\zeta \mathrm{d} Z(x^\prime) \mathrm{d} Z(x) \Bigg \rangle_{\mathcal{D}^\prime(\Sigma)} \\
      & = \lim_{\lambda \to 0^+} \Bigg \langle \chi(y)u(Z(y)), \frac{1}{(2 \pi^3)^{\frac{m}{2}}}\int_{B_{\tilde{r}/2}(0)}  \int_{\mathbb{R}\mathrm{T}^\prime_{Z(y)}(\mathcal{C}_j)} \Bigg \{ \int_{\mathbb{R}^m} e^{- \langle \zeta \rangle \left( \langle Z(x^\prime) - Z(y) \rangle^2 + \langle Z(x) + i \lambda v - Z(x^\prime) \rangle^2 \right)} \langle \zeta \rangle^{\frac{m}{2}} \Delta(Z(x^\prime) - Z(y), \zeta) \mathrm{d} Z(x^\prime) \Bigg \} \cdot \\
      & \qquad \cdot e^{i \zeta \cdot (Z(x) + i \lambda v - Z(y))}  g(Z(x)) \mathrm{d}\zeta \mathrm{d} Z(x) \Bigg \rangle_{\mathcal{D}^\prime(\Sigma)} \\
      & = \lim_{\lambda \to 0^+} \Bigg \langle \chi(y)u(Z(y)), \frac{1}{(2 \pi)^{m}}\int_{B_{\tilde{r}/2}(0)}  \int_{\mathbb{R}\mathrm{T}^\prime_{Z(y)}(\mathcal{C}_j)} e^{i \zeta \cdot (Z(x) + i \lambda v - Z(y)) - \frac{1}{2} \langle \zeta \rangle \langle Z(x) + i \lambda v - Z(y) \rangle^2}\cdot \\
      & \qquad \cdot \Delta \left(\frac{Z(x) + i \lambda v-Z(y)}{2}, \zeta \right) g(Z(x)) \mathrm{d}\zeta \mathrm{d} Z(x) \Bigg \rangle_{\mathcal{D}^\prime(\Sigma)}\\
      & = \lim_{\lambda \to 0^+} \Bigg \langle \chi(y)u(Z(y)), \frac{1}{(2 \pi)^{m}} \int_{\mathbb{R}\mathrm{T}^\prime_{Z(y)}(\mathcal{C}_j)} \mathfrak{F}^\frac{1}{2}[g](Z(y) - i \lambda v, -\zeta) \mathrm{d}\zeta \Bigg \rangle_{\mathcal{D}^\prime(\Sigma)} \\
      & = \Bigg \langle \chi(y)u(Z(y)), \frac{1}{(2 \pi)^{m}} \int_{\mathbb{R}\mathrm{T}^\prime_{Z(y)}(\mathcal{C}_j)} \mathfrak{F}^\frac{1}{2}[g](Z(y), -\zeta) \mathrm{d}\zeta \Bigg \rangle_{\mathcal{D}^\prime(\Sigma)}.
  \end{align*}
  
  \noindent We have then proved our claim. So we  conclude that
  
  \begin{equation*}
     (\chi u)|_{B_{\tilde{r}/2}(0)} = v_1|_{B_{\tilde{r}/2}(0)} + \sum_{j = 1}^N \mathrm{bv}_{\Gamma_j}(v_{2j}). 
  \end{equation*}
  
  \noindent Therefore the point $(0,\xi_0) \notin \mathrm{WF}_\mathcal{M}(u)$.
  
\end{proof}

\bibliographystyle{unsrt}
\bibliography{bibliography}
\end{document}